\documentclass[11pt,a4paper,reqno]{amsart}
  \usepackage{a4wide}
  \usepackage{amsfonts,amsmath,amssymb,amsopn,amsthm,mathrsfs,bm}
  \usepackage{framed}
  \usepackage[normalem]{ulem}
  \usepackage{soul}
  \usepackage[usenames,dvipsnames]{xcolor}
  \usepackage{tikz}\usetikzlibrary{patterns}
  \usepackage{tocvsec2}

   \def\R{\mathbb{R}}
   \def\N{\mathbb{N}}

   \def\e{{\varepsilon}}
   \def\D{{\nabla}}

   \def\vi{{\varphi}}

   \def\cA{{\mathcal A}}
   \def\cB{{\mathcal B}}
   \def\cC{{\mathcal C}}
   \def\cD{{\mathcal D}}
   
   \def\cF{{\mathcal F}}

   \def\cL{{\mathcal L}}
   \def\cM{{\mathcal M}}
   \def\cN{{\mathcal D}}

   \def\cS{{\mathcal S}}

   \def\const{\mathop{\rm const}\nolimits}
   
   \def\loc{\mathop{\rm loc}\nolimits}
   
   \def\rad{\mathop{\rm rad}\nolimits}

\newcommand{\beq}{\begin{equation}}
\newcommand{\eeq}{\end{equation}}

\usepackage{mathrsfs}
\newcommand{\tf}{{T\!F}}
\newcommand{\ch}{{\mathscr{C}}}
\renewcommand{\P}{{\mathcal P}}

\newcommand{\n}{\nabla}

\theoremstyle{definition}
\newtheorem{df}{Definition}[section]
\theoremstyle{remark}
\newtheorem{rem}[df]{Remark}
\theoremstyle{plain}
\newtheorem{prop}[df]{Proposition}
\newtheorem{lemma}[df]{Lemma}
\newtheorem{claim}[df]{Claim}
\newtheorem{teo}[df]{Theorem}

\newtheorem{cor}[df]{Corollary}

\numberwithin{figure}{section}

\numberwithin{equation}{section}

\begin{document}
\hfill{\small\sc ET-0253A-23}\vspace{5em}


\title[Normalised solutions of the GPP equation]{Normalised solutions and limit profiles\\ of the defocusing Gross--Pitaevskii--Poisson equation}

\author{Riccardo Molle}
\address{Dipartimento di Matematica, Universit\`a di Roma ``Tor Vergata'',
Via della Ricerca Scientifica n. 1, 00133 Roma, Italy}
\email{molle@mat.uniroma2.it}
\author{Vitaly Moroz}
\address{Mathematics Department, Swansea University,
Fabian Way, Swansea SA1 8EN, Wales, UK}
\email{v.moroz@swansea.ac.uk}
\author{Giuseppe Riey}
\address{Dipartimento di Matematica e Informatica, Universit\`a della Calabria,
Via P. Bucci 31B, 87036 Rende (CS), Italy}
\email{giuseppe.riey@unical.it}

\begin{abstract}
We consider normalised solutions of the stationary Gross--Pitaevskii--Poisson (GPP) equation with a defocusing local nonlinear term,
$$-\Delta u+\lambda u+|u|^2u =(I_\alpha*|u|^2)u\quad\text{in $\mathbb R^3$},\qquad\int_{\mathbb R^3}u^2dx=\rho^2,$$
where $\rho^2>0$ is the prescribed mass of the solutions, $\lambda\in\mathbb R$ is an a-priori unknown Lagrange multiplier, and $I_\alpha(x)=A_\alpha|x|^{3-\alpha}$ is the Riesz potential of order $\alpha\in(0,3)$.
When $\alpha=2$ this problem appears in the models of self--gravitating Bose--Einstein condensates, which were proposed in cosmology and astrophysics to describe Cold Dark Matter and Boson Stars.

We establish the existence of branches of normalised solutions to the GPP equation, paying special attention to the shape of the associated mass--energy relation curves and to the limit profiles of solutions at the endpoints of these curves. The main novelty in this work is in the derivation of sharp asymptotic estimates on the mass--energy curves. These estimates allow us to show that after  appropriate rescalings, the constructed normalized solutions converge either to a ground state of the Choquard equation or to a compactly supported radial ground state of the integral Thomas--Fermi equation. The behaviour of normalized solutions depends sensitively on whether $\alpha$ is greater than, equal to, or less than one.
\end{abstract}

\maketitle

\vspace{3em}
\tableofcontents
\newpage

\section{Introduction}
\settocdepth{section}

\subsection{Background}
Our starting point is the nonlocal nonlinear Schr\"odinger equation
\begin{equation}\label{GPP}
    i\partial_t\psi=-\Delta\psi+a_s|\psi|^2\psi-\big(|x|^{-1}*|\psi|^2\big)\psi, \qquad (t,x)\in \R\times\R^3,
\end{equation}
where $\psi:\R\times\R^3\to\mathbb C$, and $*$ is the convolution in $\R^3$.
Equation \eqref{GPP}, under the name of the Gross--Pitaevskii--Poisson (GPP) equation, was proposed in cosmology as a model to describes the dynamics of the Cold Dark Matter (CDM) made of axions or bosons in the form of the Bose--Einstein Condensate (BEC) at absolute zero temperatures \cite{Wang,Bohmer-Harko,Chavanis-11,Braaten,Eby}.
The nonlocal convolution term in \eqref{GPP} represents the Newtonian gravitational attraction between bosonic particles. The local cubic term accounts for the short--range quantum force self--interaction between bosons, measured by their scattering length $a_s$. The non--self--interacting case $a_s=0$ corresponds to the Schr\"odinger--Newton (Choquard) model of self--gravitating bosons \cite{Ruffini}, which is mathematically well--studied \cite{MvS-survey}. When $a_s<0$ the quantum self--interaction between bosons is focusing/attractive, while for $a_s>0$ the self--interaction is defocusing/repulsive. Introduction of the quantum self--interaction force $a_s\neq 0$ was suggested as a way to solve the small scale problems of the CDM cosmological model, such as the ``missing satellite'' problem, and the ``too big to fail'' problem (see surveys  \cite{Chavanis-15,Lee,Paredes}). Similar models appear in the literature under the names Ultralight Axion Dark Matter, and Fuzzy Dark Matter, see \cite{Lee,Braaten} for a history survey.

The standard standing--wave ansatz $\psi(t,x):=e^{i\lambda t}u(x)$ in \eqref{GPP} leads to the stationary GPP equation
\begin{equation}\label{GPP-0}
-\Delta u+\lambda u+a_s|u|^2u =\big(|x|^{-1}*|u|^2\big)u\quad\text{ in $\R^3$},
\end{equation}
where $u:\R^3\to\R$, and $\lambda\in \R$ is the frequency of the standing--wave $\psi$.
The {\em mass}
$$
\cM_\psi(t):=\int_{\R^3}|\psi(t,x)|^2dx,
$$
and the {\em energy}
$$
\cF_\psi(t):=\frac12\int_{\R^3}|\nabla\psi(t,x)|^2dx+\frac{a_s}{4}\int_{\R^3}|\psi(t,x)|^4dx-
\frac{1}{4}\iint_{\R^3\times\R^3}\frac{|\psi(t,x)|^2|\psi(t,y)|^2}{|x-y|}\, dx\, dy
$$
are conserved quantities in \eqref{GPP}.
Therefore, physically it is relevant to look for {\em normalised} stationary solutions $u_\rho$ of \eqref{GPP-0} with a prescribed mass
$$\rho^2:=\int_{\R^3}u_\rho^2dx,$$
and to analyse the {\em mass--energy} relation
$$\rho\;\mapsto\; \cF_{\psi_\rho}$$
of the normalised solutions, where $\psi_\rho(t,x):=e^{i\lambda t}u_\rho(x)$.
In this work we are interested in the normalised stationary solutions of \eqref{GPP-0} in the defocusing case $a_s>0$.
The competition between the local and non--local terms of opposite signs but of the same cubic order of nonlinearity in \eqref{GPP-0} creates interesting and nontrivial mathematical effects.
In view of the scale invariance of \eqref{GPP-0}, without loss of generality and up the unknown frequency $\lambda$, we may assume $a_s=1$.

\subsection{Framework}
Mathematically, it is meaningful to include \eqref{GPP-0} into a family of equations
\begin{equation}\label{GPP-alpha}
    -\Delta u+\lambda u+|u|^2u =(I_\alpha*|u|^2)u\quad\text{in $\R^3$}
\end{equation}
where the nonlocal term is given by the Riesz potential $I_\alpha(x):=A_\alpha|x|^{\alpha-3}$ of a variable order $\alpha\in(0,3)$. The normalisation constant $A_\alpha:=\frac{\Gamma((3-\alpha)/2)}{\pi^{3/2}2^{\alpha}\Gamma(\alpha/2)}$
ensures that $I_\alpha(x)$ could be interpreted as the Green function of the fractional Laplacian $(-\Delta)^{\alpha/2}$ in $\R^3$, and that the semigroup property $I_{\alpha+\beta}=I_\alpha*I_\beta$ holds for all $\alpha,\beta\in(0,3)$ such that $\alpha+\beta<3$, see for example \cite{Landkof}. Recall  that as $\alpha\to 0$, the Riesz potential $I_\alpha$ vaguely converges to the Dirac mass $\delta_0$ \cite{Landkof,Seok}. In particular, we note that as $\alpha=0$ the nonlinear terms cancel out and equation \eqref{GPP-alpha} formally reduces to the linear Helmholtz equation $-\Delta u+\lambda u=0$ in $\R^3$,
which has no $H^1$--solutions for any $\lambda\in\R$.

While our motivation to introduce fractional order Riesz potentials is purely mathematical, the fractional Riesz potentials may become relevant in the Modified Newtonian Dynamics (MOND) models, where modified gravitational interactions of order $\alpha\to 3$ were proposed at large galactical scales \cite{Giusti,MOND}.
Normalised  solutions of \eqref{GPP-alpha} with a prescribed mass $\rho^2>0$ in the full admissible range $\alpha\in(0,3)$ will be the main object of our study in this work.

Equation \eqref{GPP-alpha} has a variational structure, and the corresponding energy is given by
\begin{equation*}
    F(u):=\frac{1}{2}\int_{\R^3} |\D u|^2dx+\frac{1}{4}\int_{\R^3} |u|^4dx-\frac{1}{4}\cN(u),
\end{equation*}
where, and in what follows, we denote by
$$\cN(u):=\int_{\R^3}(I_\alpha*|u|^2)\, |u|^2\, dx= A_\alpha \iint_{\R^3\times\R^3}\frac{|u (x)|^2\, |u (y)|^2}{|x-y|^{3-\alpha}}\, dx\, dy$$
the nonlocal interaction energy term.
By the Hardy--Littlewood--Sobolev (HLS) inequality \eqref{HLS}, the nonlocal term $\cN(u)$ is well-defined for all $u\in L^\frac{12}{3+\alpha}(\R^3)$.
As a consequence, it is standard to show that the energy $F$ is of class $\mathcal C^1$ on $H^1(\R^3)$ (cf. \cite[Section 4.1]{LZVM}), and that for a given $\rho>0$, critical points of $F$ subject to the constraint
$$
\cS_\rho:=\left\{u\in H^1(\R^3)\ :\ \int_{\R^3} |u|^2\, dx=\rho^2\right\}
$$
correspond to the weak solutions $u$ of the problem
\beq
\label{P}\tag{$N_\rho$}
-\Delta u+\lambda u+|u|^2u =(I_\alpha*|u|^2)u,\quad u\in \cS_\rho,{ \quad \lambda\in\R},
\eeq
where $\lambda=\lambda_u\in\R$ is an a-priori unknown Lagrange multiplier. It turns out that \eqref{P} has  no sufficiently regular solutions for $\lambda_u\le 0$, see Proposition \ref{abc-rem}; and no nonnegative solutions for $\lambda_u<0$, see Lemma \ref{l-minus}.
Thus in what follows we always assume $\lambda_u>0$.
We refer to solutions of \eqref{P} as  {\em normalised solutions} of the GPP equation with a mass $\rho^2>0$.

Vice--versa, every nonzero weak solution $u$ of the fixed frequency problem
\beq
\label{Plambda}\tag{$P_\lambda$}
-\Delta u+\lambda u+|u|^2u =(I_\alpha*|u|^2)u,\quad u\in H^1(\R^3),
\eeq
with a given $\lambda>0$ corresponds to a normalised solution of \eqref{P} with an unknown mass  $\rho^2=\int|u|^2 >0$.
A {\em ground state} of $(P_\lambda)$ is a solution that has the minimal energy amongst all nontrivial solutions of $(P_\lambda)$. 
The problem \eqref{Plambda} with a prescribed parameter $\lambda\ge 0$
was recently studied amongst a more general class of Choquard type equations in \cite{LZVM}.
In particular, the following was proved in \cite[Theorem 1.1]{LZVM}.

\begin{teo}\label{thmA}
    Let $\alpha\in(0,3)$. If $\lambda=0$ then problem \eqref{Plambda} has no  solutions $u$ with $\D u\in H^1_{\loc}(\R^3)$.
    For each $\lambda>0$, equation \eqref{Plambda} admits a positive radially symmetric ground state $u_\lambda$. Moreover, every positive radially symmetric solution $u$ of \eqref{Plambda} is smooth, monotone decreasing and there exists $C_\lambda>0$ such that
    \begin{equation}\label{eq-asy}
        \lim_{|x| \to \infty} u_\lambda(x) |x|\exp \int_{\nu_\lambda}^{|x|} \sqrt{\lambda - \frac{A_\alpha\|u_\lambda\|_2^2}{s^{3-\alpha}}\,} \,ds=C_\lambda,\quad\text{where $\nu_\lambda:=\big(A_\alpha\|u_\lambda\|_2^2\lambda^{-1}\big)^\frac{1}{3 - \alpha} $}.
    \end{equation}
\end{teo}

The decay estimate \eqref{eq-asy} was established in \cite[Proposition 4.3]{LZVM} (note that the constant $\nu$ on \cite[p.160]{LZVM} has to be modified to $\nu_\lambda$ as in \eqref{eq-asy}). The same estimate holds for the Choquard equation \eqref{C1} below (cf.  \cite[Theorem 4]{MvS}). For a discussion of the implicit exponential decay estimate \eqref{eq-asy} we refer to Remark \ref{r-asy-rho} below, and to \cite[pp.157-158 and 177-178]{MvS} or \cite[Section 6.1]{MvS-JDE}.
\smallskip

The ground state $u_\lambda$ was constructed in \cite{LZVM} via minimization of the energy
$$
    F_\lambda(u):=F(u)+\frac{\lambda}{2}\int_{\R^3} |u|^2dx
$$
over the Poho\v zaev manifold of $F_\lambda$. Additionally, \cite{LZVM} studied the asymptotic profiles of the ground states $u_\lambda$ as $\lambda\to 0$ and $\lambda\to\infty$.
It was proved in \cite[Theorems 2.5, 2.7]{LZVM} and Lemma \ref{l-App-L2} below that:

\begin{itemize}
\item  as $\lambda\to 0$, the rescaled family
$$
w_\lambda(x):=\lambda^{-\frac{2+\alpha}{4}}u_\lambda\big(\lambda^{-1/2}x\big)
$$
converges in $H^1(\R^3)$ to a positive radially symmetric ground state $w_*\in H^1(\R^3)\cap C^2(\R^3)$ of the {\em Choquard equation}
\begin{equation}\label{C1}
    -\Delta w+w=(I_{\alpha}*|w|^2)w\quad\text{in $\R^3$},
\end{equation}
which was studied for $\alpha=2$ in \cite{Lieb} and for $\alpha\in(0,3)$ in \cite{MvS}; see \cite[Theorem 2.4]{LZVM} and Propositions \ref{tCmin}, \ref{pCmin} below.

\item
as $\lambda\to \infty$ and additionally assuming that $\alpha=2$, the rescaled family
$$z_\lambda(x):=\lambda^{-1/2}u_\lambda(x)$$
converges in $L^2\cap L^4(\R^3)$ to the nonnegative radially symmetric compactly supported ground state
$z_*\in C^{0,1/2}(\R^3)$ of the {\em Thomas--Fermi} type integral equation
\begin{equation}\label{eqTF1}
    z+|z|^{2}z=(I_{\alpha}*|z|^2)z\quad\text{in $\R^3$},
\end{equation}
see Proposition \ref{tTF} and discussions below.
Remarkably, since the limit Thomas--Fermi ground state
$z_*$ is compactly supported,
the rescaled ground states $z_\lambda$ develop as $\lambda\to\infty$ a steep ``corner layer'' near the boundary of the support of
$z_*$.
This phenomenon is well-known in astrophysics,
where the radius of support of the limit ground state provides an approximation to the radius
of the astrophysical object \cite{Wang,Bohmer-Harko, Chavanis-11}.
\end{itemize}

\noindent
It is clear that every fixed frequency ground state $u_\lambda$ of $\eqref{Plambda}$ is a normalised solution of \eqref{P}.
However, the {\em frequency--mass} relation of the ground state $u_\lambda$,
$$\lambda\mapsto\rho_\lambda:=\Big(\int_{\R^3}u_\lambda^2\,dx\Big)^{1/2},$$
which by Theorem \ref{thmA} is well defined for every $\lambda>0$, is a-priori unknown and nontrivial. 
We will show below (see Remark \ref{r-noninj}) that ground states of \eqref{Plambda} with different frequencies could have the same mass and different energy. In particular, a ground state of \eqref{Plambda}  may not be necessarily a normalised minimizer of \eqref{P}.
Similar phenomenon was recently observed in the local cubic--quintic NLS \cite{Killip,Lewin,Je21-1,Je21-2}, and in quasilinear Schr\"odinger equations \cite{Je15}.
We are not aware of previous results of this type in the context of nonlocal problem \eqref{P} (or similar NLS with nonlocal interactions). To complicate the matters, the uniqueness of the fixed frequency ground states of \eqref{Plambda} is entirely open.
\smallskip

In this work we use variational methods to establish the existence of branches of normalised solutions  $u_\rho$ to \eqref{P} and study asymptotic behaviour of the normalised solutions along the branches, paying special attention to the shape of the associated {\em mass--energy} $\rho\to F(u_\rho)$ and {\em mass--frequency} $\rho\to\lambda_{u_\rho}$ relation curves and to the limit profiles of solutions at the endpoints of these curves. The main novelty and difficulty in the paper is in the derivation of sharp asymptotic estimates on the mass--energy and mass--frequency relations. These estimates allow us to show that in suitable limit regimes and after  appropriate rescalings, the constructed normalized solutions converge either to a ground state of the Choquard equation or to a compactly supported radial ground state of the integral Thomas--Fermi equation. The structure and properties of normalized solutions depend nontrivially on the value of $\alpha$, with three different scenario for $\alpha>1$, $\alpha=1$ and $\alpha<1$. When $\alpha\ge 1$ the constructed normalised solutions are normalised global minimizers of $F_{|\cS_\rho}$, while for $\alpha<1$ we construct global minimizers, local minimizers and mountain--pass type critical points of $F_{|\cS_\rho}$. The asymptotic analysis of the mountain--pass critical points is particularly challenging and requires us to distingusish between mountain-pass solutions of Type I and Type II (see Remark \ref{rem-MPLI-II}). The precise statements of our results are given in Section 2. In Section 3 we outline some preliminary facts. In Sections 4, 5 and 6 we study the cases $\alpha>1$, $\alpha<1$ and $\alpha=1$ respectively.

\subsection{Notations}

\begin{itemize}
\item
    $C,c,\bar c,\hat c, c_1.\ldots$ denote various constants that can also vary from one line to another.\smallskip

    \item {\em Asymptotic notations.}
    For real valued functions $f(t), g(t) \geq 0$ on $\R_+$, we write:

    \smallskip

    $f(t)\lesssim g(t)$ at 0 (or at $+\infty$) if  $f(t) \le C g(t)$ for small $t$ (or for large $t$);\smallskip

    $f(t)\gtrsim g(t)$ if $g(t)\lesssim f(t)$;\smallskip

    $f(t)\sim g(t)$ if $f(t)\lesssim g(t)$ and $f(t)\gtrsim g(t)$;\smallskip

    $f(t)\simeq g(t)$ if $f(t) \sim g(t)$ and $\lim\frac{f(t)}{g(t)}=1$ as $t\to 0$ (or $t\to\infty$).\smallskip

    \noindent
    We write $f(t)=o(g(t))$ if $ \frac{f(t)}{g(t)}\to 0$ as $t\to 0$ (or $t\to +\infty$).

    \noindent
    Analogous definitions will be used for real valued functions $f(t), g(t) \leq 0$.
    \smallskip

    \item
    $L^p=L^p(\R^3)$ denotes the Lebesgue space on $\R^3$ with $1\leq p<+\infty$.
    The norm in $L^p$ is denoted by $|\cdot|_{p}$.
    If $\Omega\subsetneq\R^3$ we always use the notation $L^p(\Omega)$.
    \smallskip

    \item
    $H^1=H^1(\R^3)$ denotes the Sobolev space on $\R^3$. The standard norm in $H^1$ is denoted by $\|\cdot\|$.
    By $H^1_{\rad}$ we denote the space of radial functions in $H^1$, and
    $$
    \cS_{\rho,\rad}:=\cS_\rho\cap H^1_{\rad}.
    $$

\end{itemize}

\section{Main results}

\subsection{Scaling considerations and limit regimes.}
Before we formulate our main results, we want to introduce several scaling considerations and to describe Choquard and Thomas--Fermi problems, which provide limit profiles for \eqref{P}.

In what follows we are going to construct normalised solutions of \eqref{P} as local minimizers, or more generally, critical points of the energy $F$ subject to constraint on $\cS_\rho$. Denote by
\begin{equation}\label{eq-mrho}
    m_\rho:=\inf_{\cS_\rho}F=\inf\Big\{F(u)\ :\ u\in H^1, \ |u|_2=\rho\Big\}
\end{equation}
the minimal energy level on $\cS_\rho$. Since the energy $F$ contains three different power--like terms that scale at different rates, the relation between $\rho$ and $m_\rho$ is implicit, and we will see below that in some cases $m_\rho$ is not achieved.

On the other hand when an energy includes two terms only, due to the scaling invariance, the mass--energy and mass--frequency relation is an explicit function of the mass. We consider this in the case of {\em Choquard} and {\em Thomas--Fermi} energies.

\subsubsection*{\bf Choquard limit for $\alpha\neq 1$}
Consider the Choquard functional
$$E_\ch(u):=\frac{1}{2}|\nabla u|_2^2-\frac14\cN(u),$$
and the minimization problem
\begin{equation}\label{eqC}
    m^\ch_\rho:=\inf_{\cS_\rho} E_\ch.
\end{equation}
The case $\alpha=1$ is $L^2$--critical for \eqref{eqC} and will be considered separately, see Proposition \ref{pCmin}.

For $\alpha\neq 1$ and $\rho>0$ consider the scaling
\begin{equation}\label{eq-beta1}
    \cS_\rho\ni u(x)\;\mapsto\; w(x):=\rho^{-\frac{\alpha+2}{\alpha-1}} u\big(\rho^{-\frac{2}{\alpha-1}} x\big)\in \cS_1.
\end{equation}
If $\alpha>1$ then
$$m^\ch_\rho=\rho^{2\frac{\alpha+1}{\alpha-1}} m^\ch_1,$$
and if $w_0\in\cS_1$ is a minimizer for $m^\ch_1$ then $u_\rho(x)=\rho^{\frac{\alpha+2}{\alpha-1}} w_0\big(\rho^{\frac{2}{\alpha-1}} x\big)$ is a minimizer for $m^\ch_\rho$.

If $\alpha<1$ then $m^\ch_\rho=-\infty$ for all $\rho>0$ in view of the mass preserving scaling $t^{3/2}u(tx)$, but the same mass--energy scaling consideration applies to the higher energy critical points of $E_\ch$ on $\cS_1$.

Next we observe that if $u\in \cS_\rho$ and $u\mapsto w\in \cS_1$ via rescaling \eqref{eq-beta1} then
\beq\label{Ch-scale1}
F(u)=E_\ch(u)+\tfrac14|u|_4^4
=\rho^{2\frac{\alpha+1}{\alpha-1}}\Big(E_\ch(w)+\tfrac{1}{4}\rho^{\frac{2\alpha}{\alpha-1}}|w|_4^4\Big).
\eeq
This suggests that for $\alpha>1$ and $\rho\to 0$, or $\alpha<1$ and $\rho\to\infty$, the $L^4$--term can be neglected and at least formally, we may consider $\rho^{2\frac{\alpha+1}{\alpha-1}}E_\ch|_{\cS_1}$ as a limit functional for $F|_{\cS_\rho}$.

The following result about the properties of $E_\ch|_{\cS_1}$ is well--known, see \cite{Lieb} for $\alpha=2$, or \cite{Ye,Ye2} for fractional $\alpha\in(0,3)$ and \cite{MvS} for the qualitative properties of the ground states.

\begin{prop}\label{tCmin}
    $(i)$ If $\alpha\in(1,3)$ then $m^\ch_1<0$ and there exists a positive radially symmetric minimizer $w_0\in \cS_1$ such that $E_\ch(w_0)=m^\ch_1$.

    $(ii)$ If $\alpha\in(0,1)$ then $m^\ch_\rho=-\infty$ for all $\rho>0$, and $E_\ch|_{\cS_1}$ has a radially symmetric mountain pass critical point $v_0>0$ such that
    \beq
    \label{N2.2}
    M^\ch_1:=E_\ch(v_0)>0.
    \eeq
\end{prop}

Recall also that every critical point $w$ of $E_\ch|_{\cS_1}$ at an energy level $\mu\in\R$ is a  solution of the Choquard equation
\begin{equation}\label{Ch-01}
-\Delta w+\lambda_\ch w=(I_\alpha*w^2\big)w, \quad w\in \cS_1,
\end{equation}
where $\lambda_\ch>0$ is a Lagrange multiplier (that depends on $w$), and
\beq
\label{DR}
\lambda_\ch=\frac{2(1+\alpha)}{1-\alpha}\mu,\quad |\nabla w|_2^2=\frac{2(3-\alpha)}{1-\alpha}\mu,\quad  \cN(w)=\frac{8}{1-\alpha} \mu.
\eeq
Every positive solution of \eqref{Ch-01} is smooth, strictly monotone decreasing and has exponential decay at infinity \cite{MvS}.

\begin{rem}\label{r-GN}
    Alternatively, one can look for mininimizers of the Gagliardo-Nirenberg quotient
    \beq
    \label{N2.3}
    S_\alpha:=\inf_{u\in H^1(\R^3\setminus\{0\})}\frac{|\D u|_2^{3-\alpha}\,|u|_2^{1+\alpha}}{\cD(u)},
    \eeq
    which is invariant w.r.t. dilations and scalar multiplications. It is well known (cf. \cite{Ye}) that $S_\alpha$ is achieved for any $\alpha\in(0,3)$, minimizers for \eqref{N2.3} satisfy the Euler--Lagrange equation
    \begin{equation}\label{Ch-GN}
        -\frac{3-\alpha}{|\D u|_2^2}\Delta u+ \frac{1+\alpha}{|u|_2^2}u=\frac{4}{\cD(u)}(I_\alpha*u^2\big)u.
    \end{equation}
 and normalised solutions $w_0$ and $v_0$ constructed in Proposition \ref{tCmin} are amongst the minimizers for $S_\alpha$ (cf. \cite[Proposition 2.1]{MvS} or \cite{Ye}). Using the relations
\eqref{DR}, we realise that \eqref{Ch-GN} reduces to \eqref{Ch-01} if $u=v_0$ or $u=w_0$.
\end{rem}

\begin{rem}
\label{UC}
The uniqueness of the positive normalised minimizer of \eqref{Ch-01} for $\alpha=2$ was proved in \cite{Lieb}. For general $\alpha$ the uniqueness is an open problem. Several partial results are known, cf. \cite{Seok} where the uniqueness is proved for $\alpha\to 0$ and $\alpha\to 3$.
\end{rem}

\subsubsection*{\bf Thomas--Fermi limit}
Consider the Thomas--Fermi functional
$$
E_\tf(u):=\frac{1}{4}|u|_4^4-\frac14\cN(u)
$$
and the minimization problem
\begin{equation}\label{eqTF2}
m^\tf_\rho:=\inf\Big\{E_\tf(u)\ :\ u\in L^2\cap L^4, \ |u|_2=\rho\Big\},
\end{equation}
and observe that
$\cN(u)\le c_\alpha\rho^{\frac{4}{3}\alpha} \,|u|_4^{4-\frac{4}{3}\alpha}$ for all
$u\in \cS_\rho$ in view of \eqref{stima-CB},
so that the minimization problem is well defined and $m^\tf_\rho>-\infty$ for all $\rho>0$.

It is clear from the elementary scaling
\begin{equation}\label{eq-betaTF}
\cS_\rho\ni u(x)\;\mapsto\; z(x):=\rho^{-1} u\big(x\big)\in \cS_1
\end{equation}
that
$$
m^{T\!F}_\rho=\rho^4 m^{T\!F}_1,
$$
and that if
$z_*$ is a minimizer for $m^{T\!F}_1$ then
$\rho z_*$ is a minimizer for $m^{T\!F}_\rho$, for any $\rho>0$.

Next observe that if $u\in \cS_\rho$ and
$z=\rho^{-1}u\in \cS_1$ then
$$
F(u)=\tfrac{1}{2}|\nabla u|_2^2+E_\tf(u)=\rho^4\Big(E_\tf(z)+\tfrac{1}{2}\rho^{-2}|\nabla z|_2^2\Big),
$$
which suggest that for $\rho\to\infty$ the gradient term can be neglected and at least formally, we may consider
$\rho^{4}E_\tf|_{\cS_1}$ as a limit functional for $F|_{\cS_\rho}$, and to expect that
$$
m_\rho\simeq \rho^4 m^\tf_1.
$$
The minimization problem $m^\tf_1$ is well-understood.
The following result for $\alpha=2$ is classical and goes back to \cite{Auchmuty-Beals}, see also \cite{Lions-81}.
For general $\R^N$ and fractional range $\alpha\in(0,N)$ see \cite[Section 3]{Carrillo-CalcVar}
and further references therein.
The uniqueness for general $\alpha\le 2$ is due to \cite[Lemma 5.2 and Proposition 5.13]{Volzone}
while for $\alpha>2$ the uniqueness is due to \cite{Carrillo-unique}.

\begin{prop}\label{tTF}
    For every $\alpha\in(0,3)$, $m^\tf_1<0$ and there exists a unique nonnegative minimizer
    $z_\infty$ for $m^\tf_1$.
    The minimizer
    $z_\infty$ is radially symmetric and radially nonincreasing. Moreover,
    $z_\infty\in C^{0,1/2}$,
    $I_\alpha*z_\infty^2\in C^{0,1}$,
    $\mathrm{supp}(z_\infty)=\bar B_{R_\infty}$ for some $R_\infty<\infty$ and
    $z_\infty$ is smooth inside the support.
    The minimizer $z_\infty$ satisfies the Thomas--Fermi type integral equation
    \begin{equation}
    \label{TF-rho-r}
    \big(|z_\infty|^2-I_\alpha*|z_\infty|^2\big)z_\infty=4m^\tf_1 z_\infty \quad\text{in $\R^3$}.
    \end{equation}
    Moreover,
    $|z_\infty|_4^4=-4\frac{3-\alpha}{\alpha}m^\tf_1$,
    $\cN(z_\infty)=-\frac{12}{\alpha} m^\tf_1$.
\end{prop}

\begin{rem}
    Note that in most of the previous works the problem \eqref{eqTF2} is formulated in $L^1\cap L^2$
    in terms of the unknown function
    $\phi=|z|^2\ge 0$, e.g.
    $$
    4m^\tf_1=\inf\Big\{|\phi|_2^2-\cN(\sqrt{\phi})\ :\  0\le \phi\in L^1\cap L^2, \ |\phi|_1=1\Big\},
    $$
    which is equivalent to \eqref{eqTF2} from the point of view of the minimization process.
    In view of the connection to \eqref{P} for us it is convenient to formulate \eqref{eqTF2} in terms of
    $z$ in $L^2\cap L^4$, see also \cite[Sections 2.5 and 6]{LZVM}.
    The  difference with the $L^1\cap L^2$ formulation is that the Euler--Lagrange equation \eqref{TF-rho-r} is written differently.
    Note however that \eqref{TF-rho-r} is straightforward to deduce since we do not need to introduce the
    $\{z\ge 0\}$ constraint in \eqref{eqTF2}.
    Since
    $E_\tf(z)=E_\tf(|z|)$,
    the existence of a nonnegative minimizer is ensured a-posteriori by considering the absolute value of any minimizer.
    The compact support of the minimizer is controlled in \eqref{TF-rho-r} by the additional multiplier
    $z_\infty$, which is significant only outside of the support of $z_\infty$.
    If we set $\phi_\infty:=|z_\infty|^2$ then \eqref{TF-rho-r}   can be seen to be  equivalent to
    \begin{equation}\label{eqTFsupp}
    \phi_\infty=\big(I_\alpha*\phi_\infty+4m^\tf_1\big)_+ \quad\text{in $\R^3$},
    \end{equation}
    which is precisely the classical Thomas--Fermi equation for $E_\tf(\sqrt{\phi})$.
    Note that while  \eqref{eqTFsupp} has a linear structure inside the support of
    $z_\infty$, this is not a linear equation but a free boundary problem, since the radius of the support is a part of the unknown.
\end{rem}

\begin{rem}
One more formal rescaling possibility is to consider the {\em defocusing Gross--Pitaevskii functional}
$$E_{GP}(u):=\frac{1}{2}|\nabla u|_2^2+\frac{1}{4}|u|_4^4.$$
The relevant rescaling is
\begin{equation}\label{eq-beta3}
\cS_\rho\ni u(x)\;\mapsto\; w(x):=\rho^2 u\big(\rho^2 x\big)\in \cS_1,
\end{equation}
which formally leads to a representation
$$F(u)=\rho^{-2}E_{GP}(w)-\tfrac14\rho^{2(\alpha-1)}\cN(w)=\rho^{-2}\Big(E_{GP}(w)-\tfrac14\rho^{2\alpha}\cN(w)\Big).$$
However $E_{GP}$ has no critical points on $\cS_1$ and the defocusing Gross--Pitaevskii functional does not bring any useful information into our considerations.
\end{rem}

\subsection{Main results}
Now we are in a position to formulate our main results.
First of all we consider the case $\alpha\in(1,3)$, which includes the astrophysical scenario $\alpha=2$.

  \begin{teo}
  \label{T11}
  Assume $\alpha\in(1,3)$ and $\rho>0$.
  Then $m_\rho<0$ and the map $\rho\mapsto m_\rho$ is strictly decreasing and strictly concave.
  For every $\rho>0$ there exists a positive radially symmetric and monotone decreasing normalised solution $u_\rho$ of \eqref{P}, which is a minimizer of $F|_{\cS_\rho}$, so that $F(u_\rho)=m_\rho$,   and  $u_\rho\in L^1 \cap C^2\cap W^{2,s}$, for every $s>1$.
  In addition, the corresponding to $u_\rho$ Lagrange multiplier $\lambda_\rho>0$, and
  \begin{itemize}
  \item[$(i)$]
  as $\rho\to 0$,
  \beq
  \label{luc}
  m_\rho\simeq m^\ch_1\rho^{2\frac{\alpha+1}{\alpha-1}},\qquad\lambda_\rho \simeq -2\frac{\alpha+1}{\alpha-1} m^\ch_1 \rho^{\frac {4}{\alpha-1}},
  \eeq
  and, for every $\rho_n\to 0$, the rescaled family
 \beq\label{Ch-scale}
 w_{\rho_n}(x):=\rho_n^{-\frac{\alpha+2}{\alpha-1}} u_{\rho_n}\left(\rho_n^{-\frac{2}{\alpha-1}}x\right)
 \eeq
has a subsequence converging in $H^1$ to a positive radially symmetric Choquard minimizer $w_0\in \cS_1$ of $E_{\ch}\big|_{\cS_1}$ such that $E_\ch(w_0)=m^\ch_1$.
  \smallskip

  \item[$(ii)$]
  As $\rho\to \infty$,
  \beq\label{TFasymp}
  m_\rho\simeq m^\tf_1\rho^4, \qquad \lambda_\rho \simeq -4m^\tf_1\rho^2,
  \eeq
  and the rescaled family
  \begin{equation}\label{eq2res}
    z_\rho(x):=\rho^{-1}u_\rho(x)
  \end{equation}
  converges in $L^2$ and $L^4$ to the nonnegative radially symmetric compactly supported Thomas--Fermi minimizer $z_\infty\in C^{0,1/2}$ of
  $E_{\tf}\big|_{\cS_1}$ such that $E_\tf(z_\infty)=m^\tf_1$.
  \end{itemize}
 \end{teo}

\begin{rem}
If \eqref{Ch-01} has a unique solution $w_0$ (see Remark \ref{UC}), then $(i)$ implies that $w_{\rho}\to w_0$, in $H^1$.
\end{rem}

\begin{rem}\label{rGPP}
    The convergence to the limit Thomas--Fermi profile for large masses in $(ii)$
    is precisely the phenomenon that was already observed for $\alpha=2$ in the astrophysical literature \cite{Wang,Bohmer-Harko}.
    In fact, in the case $\alpha=2$ the Thomas--Fermi energy $m^\tf_1$ and the minimizer
    $z_\infty$ are explicitly known,
    $$m^\tf_1=-\tfrac{3}{16\pi^2},\qquad z_\infty(|x|)=\begin{cases}
    \frac{\sqrt{3}}{2\pi}\sqrt{\frac{\sin(|x|)}{|x|}},&|x|<\pi,\smallskip\\
    0,&|x|\ge\pi,
    \end{cases}
    $$
    see \cite[p.92]{Chandrasekhar} or \cite[Remark 3.1]{LZVM} (see also \cite{Wang,Bohmer-Harko}).
    In this context $z_\infty$ (up to the physical constants)
    is the Thomas--Fermi approximate solution for the self--gravitating BEC,
    as observed in \cite{Wang,Bohmer-Harko,Chavanis-11}.
    The support radius $R_\infty=\pi$ provides an approximate radius of the astrophysical object.
    Note that $z_\infty\not\in H^1$!
\end{rem}

\begin{rem}\label{r-massR}
    The {\em mass--radius} relation is an important quantity in astrophysics which for a given $\theta\in(0,1)$ is defined as the radius of the ball $B_R$ which contains $\theta$ of the total mass $\rho^2$ of the astrophysical object (centered at the origin).
    Using the convergence of the rescaled family \eqref{Ch-scale} to a limit Choquard profile $w_0\in \cS_1$, up to a subsequence, it is simple to derive the asymptotic mass--radius relation for small masses $\rho_n^2\to 0$.
    Indeed, let $R_0>0$ be such that
    $$\int_{B_{R_0}}w_0^2\, dx=\theta.$$
    Then
    $$\int_{B_{R_0\rho_n^{-\frac{2}{\alpha-1}}}}u_{\rho_n}^2\, dx\simeq\theta\rho_n^2,$$
    that gives the mass--radius relation $R_{\rho_n}\simeq R_0\rho_n^{-\frac{2}{\alpha-1}}$ as $\rho_n\to 0$.
    For $\alpha=2$ this is consistent with the asymptotic expression $(93)$ in \cite{Chavanis-11}.   
\end{rem}

\begin{rem}\label{r-asy-rho}
The decay estimate \eqref{eq-asy}, combined with asymptotic estimates on the Lagrange multipliers $\lambda_\rho$ in Theorem \ref{T11} allow to derive decay estimates for the minimizers $u_\rho$ as $|x|\to\infty$.
Recall (see  \cite[Remark 6.1]{MvS-JDE} and \cite[Remark 6.1]{MvS}) that if $\alpha < 2$ then \eqref{eq-asy} gives the same standard exponential decay rate as for the Green function of the Schr\"odinger operator $-\Delta+\lambda_\rho$, i.e.
\beq
u_\rho(x)\simeq c_\rho|x|^{-1}e^{-\sqrt{\lambda_\rho}|x|}.
\eeq
On the other hand, when $\alpha \ge 2$, the effect of the nonlocal term is strong enough to modify the asymptotics of the groundstates. When $\alpha = 2$ (and $A_2=\frac1{4\pi}$) the nonlocal term creates a polynomial correction to the standard exponential asymptotics (\cite[Remark 6.1]{MvS-JDE} and \cite[Remark 6.1]{MvS}),
so that
\beq
u_\rho(x)\simeq c_\rho |x|^\frac{\rho^2}{8\pi\sqrt{\lambda_\rho}} \Big\{|x|^{-1}e^{-\sqrt{\lambda_\rho}|x|}\Big\}.
\eeq
When $\alpha \in (2,\frac52)$ the correction is an exponential (\cite[Remark 6.1]{MvS-JDE} and \cite[Remark 6.1]{MvS}),
\beq
u_\rho(x)\simeq  c_\rho e^{\frac{A_\alpha\rho^2} {2 (\alpha-2)\sqrt{\lambda_\rho}}|x|^{\alpha-2}}
\Big\{|x|^{-1}e^{-\sqrt{\lambda_\rho}|x|}\Big\}.
\eeq
Larger values of $\alpha\in[\frac{5}{2},3)$ could be analysed by taking higher-order Taylor expansions of the square root as in \cite[Remark 6.1]{MvS-JDE}.
\end{rem}

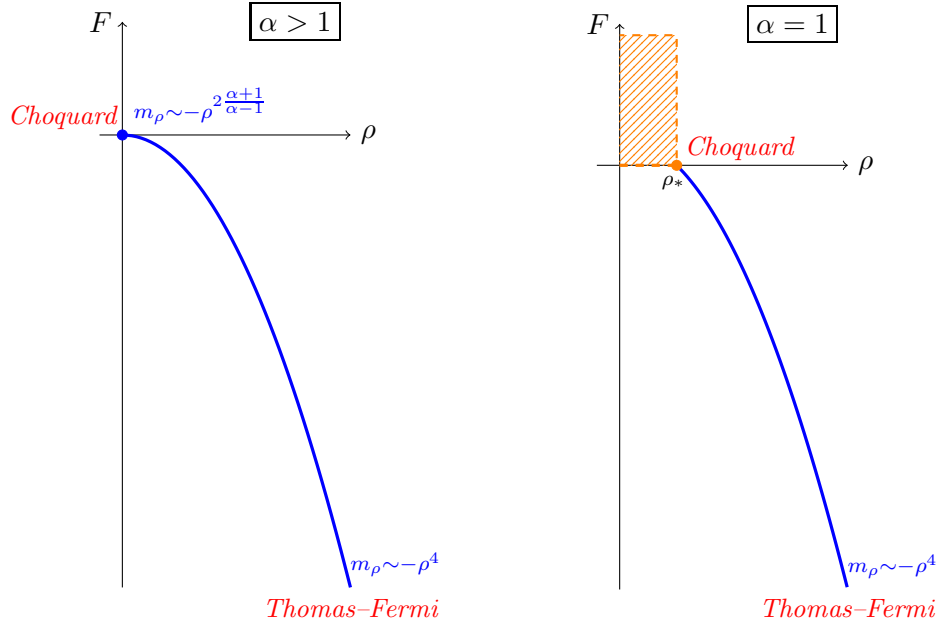
\begin{figure}
    \centering
    \begin{tikzpicture}[scale=1.5]
    \draw[->] (-0.2, 0) -- (2, 0) node[right] {$\rho$};
    \draw[->] (0, -4) -- (0, 1) node[left] {$F$} node[right] {$\qquad\qquad\boxed{\alpha>1}$};
    \draw[scale=1, domain=-0:2, smooth, variable=\x, blue, very thick] plot ({\x}, {-\x*\x})
    node[below] {{\color{red}\small\em Thomas--Fermi}}
    (1.9,-4) node[above right] {${\scriptstyle m_\rho\sim -\rho^{4}}$} (0.05,-0.05) node[above left] {{\color{red}\small\em Choquard}}
    (0,0) node[above right] {${\scriptstyle m_\rho\sim-\rho^{2\frac{\alpha+1}{\alpha-1}}}$};
    \filldraw [blue] (0.0,0.0) circle (1.25pt);
    \end{tikzpicture}
    \qquad\qquad
    \begin{tikzpicture}[scale=1.5]
    \draw[->] (-0.2, 0) -- (2, 0) node[right] {$\rho$};
    \draw[->] (0, -3.75) -- (0, 1.25) node[left] {$F$} node[right] {$\qquad\qquad\boxed{\alpha=1}$};
    \draw[scale=1, domain=0.5:1.995, smooth, variable=\x, blue, very thick] plot ({\x}, {0.25-\x^2})
    node[below] {{\color{red}\small\em Thomas--Fermi}}
    (1.9,-3.75) node[above right] {${\scriptstyle m_\rho\sim -\rho^{4}}$}
    (1.05,-0.05) node[above] {{\color{red}\small\em Choquard}};
    \filldraw [orange] (0.5,0.0) circle (1.25pt) node[below] {{\color{black}${\scriptstyle\rho_*\;}$}};
    \draw[thick,dashed,orange,pattern=north east lines,pattern color=orange] (0,0) -- (0.5,0)--(0.5,1.15)--(0,1.15);
    \draw[scale=1, domain=0:0.5, smooth, variable=\x, orange, very thick, dashed] plot ({\x}, {0*\x});
    \end{tikzpicture}
\caption{Sketch of the global minimum branch $m_\rho$ for $\alpha\in(1,3)$ and $\alpha=1$. Dashed region for $\alpha=1$ indicates nonexistence.}
\end{figure}

The case $\alpha=1$ is $L^2$--critical for the Choquard energy,
that is $E_\ch$ is homogeneous under the mass--preserving scaling $w_t=t^{3/2}w(tx)$:
$$
E_\ch(w_t)=\left(\tfrac{1}{2}|\nabla w|_2^2-\tfrac14\cN(w)\right)\cdot t^2 \qquad \forall t>0.
$$
Let $w_*$ be a $1$-frequency ground state of the Choquard equation, that is a solution of
\beq\label{ChL2}
-\Delta w+w=(I_1*|w|^2)\,w,\quad u\in H^1.
\eeq
The existence of at least one ground state $w_*>0$ of \eqref{ChL2} was proved in \cite{MvS}.
The uniqueness of ground states \eqref{ChL2} is in general open.
Each ground state of \eqref{ChL2} is smooth, positive (or negative), radially symmetric, monotone decreasing and has an exponential decay at infinity \cite{MvS}.
All ground states of \eqref{ChL2} have the same $L^2$-norm (see \cite[Lemma 3.2]{Ye}).
In particular, the quantity
$$\rho_*:=|w_*|_2>0$$
is uniquely defined. The constant $\rho_*$ is related to the optimal constant in the Gagliardo-Nirenberg inequality \eqref{N2.3} with $\alpha=1$, namely
\begin{equation}\label{GNCh}
\frac{\rho_*^{2}}{2}=S_1:=\inf_{0\neq w\in H^1}\frac{|\D w|_2^{2}\,|w|_2^2}{\cN(w)},
\end{equation}
where the infimum is achieved if and only if $w$ is a ground state of \eqref{ChL2}.
The following result is well--known, see \cite[Theorem 1.1]{Ye} for a detailed exposition.

\begin{prop}\label{pCmin}
    If $\alpha=1$ then
    $$m^\ch_\rho=
    \begin{cases} 0,  & 0<\rho\le \rho_*,\\
    -\infty, & \rho>\rho_*,\end{cases}$$
    and each ground state of \eqref{ChL2} is a minimiser for $m^\ch_{\rho_*}$, while $m^\ch_{\rho}$ has no minimisers if $\rho<\rho_*$.
\end{prop}

Moreover, if $w_*$ is a $1$-frequency minimizer for $m^\ch_{\rho_*}$ then
\beq
\label{1123}
\cD(w_*)=2|\D w_*|_2^2=2\rho_*^2,
\eeq
$w_{*,t}(x):=t^{3/2}w_*(tx)\in \cS_{\rho_*}$ is also a minimizer for $m^\ch_{\rho_*}$ for every $t>0$, and $w_{*,t}$ satisfies the rescaled equation
\beq\label{ChL2t}
-\Delta w+t^2 w=(I_1*w^2)w,\quad w\in \cS_{\rho_*}.
\eeq
It turns out that the quantity $\rho_*$ controls the behaviour of normalized solutions for \eqref{P}.

\begin{teo}
    \label{T13}
    Assume $\alpha=1$ and $\rho>0$. Let $\rho_*:=|w_*|_2$, where $w_*\in H^1$ is a positive radially symmetric ground state solution of the $L^2$--critical Choquard equation \eqref{ChL2}. Then
    $$m_\rho\;
    \begin{cases} \;=0,  & 0<\rho\le \rho_*,\\
    \;<0, & \rho>\rho_*,\end{cases}$$
    and for $\rho\ge\rho_*$ the map $\rho\mapsto m_\rho$ is strictly decreasing and strictly concave.
    If $\rho\le\rho_*$ then the infimum $m_\rho=\inf_{\cS_\rho}F=0$ is not attained
    and $F$ has no critical points on $\cS_\rho$ such that $\nabla u\in H^1_{loc}$,
    while for every $\rho>\rho_*$ there exists a positive radially symmetric normalised solution
    $u_\rho$ of \eqref{P}, which is a minimizer of $F|_{\cS_\rho}$, so that $F(u_\rho)=m_\rho$.

    In addition, the corresponding to $u_\rho$ Lagrange multiplier $\lambda_\rho>0$, and
    \begin{itemize}
        \item[$(i)$]
        As $\rho\to \rho_*$ and $\rho>\rho_*$,
        $$
        0>m_\rho\ge -C \rho^4\Big(1-\big(\tfrac{\rho_*}{\rho}\big)^2\Big)\gtrsim-(\rho-\rho_*)\to 0,
        \qquad {  0<\lambda_\rho \lesssim} -m_\rho^\frac{2}{3}\to 0,
        $$
        moreover, for every sequence $ \rho_n\to 0$ the rescaled family
        \beq\label{rescaling-lambda}
        v_{\lambda_{\rho_n}}(x)=\lambda_{{\rho_n}}^{-3/4}u_{{\rho_n}}\big(\lambda_{\rho_n}^{-1/2}x\big)
        \eeq
        has a subsequence converging in $H^1$ to $w_*\in \cS_{\rho_*}$ such that $E_\ch(w_*)=0$.
        \smallskip

        \item[$(ii)$] As $\rho\to\infty$, the statement $(ii)$ of Theorem \ref{T11} remains valid.
    \end{itemize}

\end{teo}

Unlike our other results, here the rescaling of the normalised minimizer $u_\rho$ that ensures the convergence to the Choquard limit as $\rho\to\rho_*$ is an implicit function of $\rho$, controlled in terms of the Lagrange multiplier $\lambda_\rho$.
This is natural, since the limit problem $m_{\rho_*}^\ch$ has a continuum branch of minimizers
$w_{*,t}=t^{3/2}w_*(tx)\in \cS_{\rho_*}$ with variable Lagrange multiplier $t^2$. Thus we need to fix a specific minimizer within the family, before we study convergence to the Choquard limit. The rescaling \eqref{rescaling-lambda}
is chosen to ensure convergence to $w_{*,1}=w_*$.

 \begin{figure} \label{F12}
    \centering
    \begin{tikzpicture}[scale=1.5]
        \draw[->] (-0.2, 0) -- (3, 0) node[below] {$\rho$};
        \draw[->] (0, -4) -- (0, 1.25) node[left] {$F$} node[right] {$\qquad\qquad\boxed{\alpha<1}$};
        \draw[scale=1, domain=0.5:2.235, smooth, variable=\x, blue,very thick] plot ({\x}, {1-\x^2})
        node[below] {{\color{red}{\small\it Thomas--Fermi}}}
        (2.1,-3.75) node[above right] {${\scriptstyle m_\rho\sim -\rho^{4}}$}
        (2.9,0.1) node[right] {{\color{red}{\small\it Choquard}}}
        (2.7,0.15) node[above] {{\color{purple}${\scriptstyle M_\rho\sim \rho^{-2\frac{1+\alpha}{1-\alpha}}}$}};
        \draw[scale=1, domain=0.5:3, smooth, variable=\x, purple, very thick] plot ({\x}, {(0.38/\x)});
        \draw[scale=1, domain=2.0:3, smooth, variable=\x, yellow, very thick, dotted] plot ({\x}, {(0.38/\x)});
        \draw[scale=1, domain=0:0.75, smooth, variable=\y, black, thick, dotted] plot ({0.5}, {\y});
        \filldraw [purple] (0.5,0.75) circle (1.25pt) node[above] {{\color{purple}\Large\bf ?}};
        \draw[thick,dashed,orange,pattern=north east lines,pattern color=orange] (0,0) -- (0.3,0)--(0.3,1.2)--(0,1.2)
        (0.3,-0.0224) node[below] {${\color{black}{\scriptstyle\bar\rho}}$}
        (0.6,0.0) node[below] {${\color{black}{\scriptstyle\rho^{**}}}$};
        \draw[scale=1, domain=0:1.0, smooth, variable=\x, orange, very thick, dashed] plot ({\x}, {0*\x});
        \filldraw [blue] (1.0,0) circle (1.25pt) node[below] {${\color{black}{\scriptstyle\rho^*\;\;\;}}$};
        \filldraw [blue] (2.0,0) circle (0.5pt) node[below] {${\color{black}{\scriptstyle\overline{\rho}^{**}}}$};
    \end{tikzpicture}
    \caption{Sketch of the minimum branch $m_\rho$ and mountain pass branch $M_\rho$ for $\alpha\in(0,1)$. Dashed region indicates nonexistence. The behaviour of the branches near $\rho^{**}\ge\bar\rho$ is a conjecture only. The dotted part of the upper branch for $\rho>\overline{\rho}^{**}$ corresponds to the mountain pass solutions of type II.}
\end{figure}
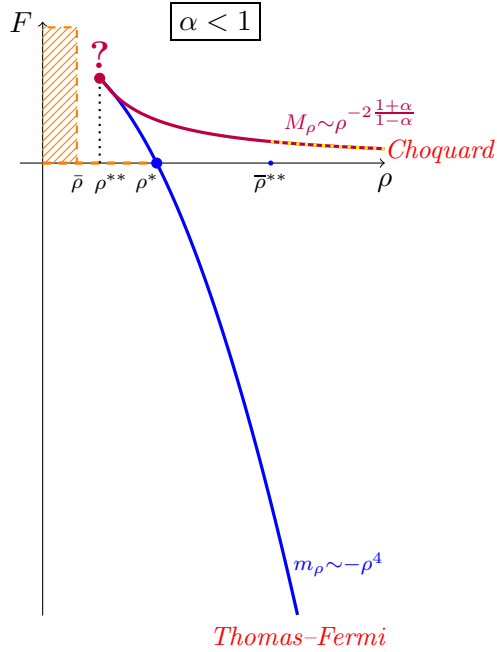

\medskip

The case $\alpha<1$ is $L^2$--supercritical for the Choquard energy $E_\ch$.
The behaviour of the branches of normalised solutions of \eqref{P} is more complicated, see Figure \ref{F12}.
There are no normalised solutions for small $\rho<\overline{\rho}$, while for sufficiently large $\rho>\rho^{**}$ we distinguish the lower branch, which consists of local or global minima of $F|_{\cS_{\rho},\rad}$, and the upper branch which consists of mountain pass type solutions of $F|_{\cS_{\rho},\rad}$. By a mountain pass type solution in the rest of the paper we understand a critical point which is not a local minimum. The precise statement is as follows.

  \begin{teo}
    \label{T12}
    Assume $\alpha\in(0,1)$ and $\rho>0$.

    $(a)$ There exists $\bar\rho>0$ such that  for $\rho\in(0,\bar\rho)$ the functional $F|_{\cS_\rho}$ has no critical points $u$ such that $\D u\in H^1_{\loc}$.

    $(b)$ There exists $\rho^*>\bar \rho$ such that
    $$m_\rho\;
    \begin{cases} \;=0,  & 0<\rho\le \rho^*,\\
    \;<0, & \rho>\rho^*,\end{cases}$$
    and for $\rho\ge\rho^*$ the map $\rho\mapsto m_\rho$ is strictly decreasing and strictly concave.
    If $\rho<\rho^*$ then the infimum $m_\rho=\inf_{\cS_\rho}F=0$ is not attained, while for every $\rho\ge\rho^*$ there exists a positive radially symmetric normalised solution $u_\rho$ of \eqref{P}, which is a minimizer of $F|_{\cS_\rho}$, so that $F(u_\rho)=m_\rho$.

    $(c)$ There exists $\rho^{**}\in[\bar\rho,\rho^*)$ such that
    \begin{itemize}
        \item
        for all $\rho\in(\rho^{**},\rho^*)$ there exists a positive radially symmetric normalised solution $u_\rho$ of \eqref{P}, which is a local minimum of $F|_{\cS_{\rho},\rad}$ and $F(u_\rho)=:\bar m_\rho>0$,\smallskip

        \item
        for all $\rho>\rho^{**}$ there exists a positive radially symmetric normalised solution $v_\rho$ of \eqref{P}, which is a mountain pass type critical point of $F|_{\cS_{\rho,\rad}}$
        and $F(v_\rho)=:M_\rho>\max\{\bar m_\rho,0\}$.

    \end{itemize}
The solutions $u_\rho,v_\rho\in L^1 \cap C^2\cap W^{2,s}$, for every $s>1$.
In addition,  the corresponding to $u_\rho,v_\rho$ Lagrange multipliers $\lambda_\rho,\bar\lambda_\rho>0$, and
    \begin{itemize}

            \item[$(i)$]
            as $\rho\to\infty$,
            \beq\label{asymptMPL-2}
            M_\rho\simeq \rho^{-2\frac{\alpha+1}{1-\alpha}}M^\ch_1,\qquad
            \bar\lambda_\rho \simeq 2\frac{1+\alpha}{1-\alpha} M^\ch_1 \rho^{-\frac{4}{1-\alpha}},
            \eeq
            moreover, for every sequence $\rho_n\to\infty$ the rescaled family
            \beq
            w_{\rho_n}(x):=\rho_n^{\frac{\alpha+2}{1-\alpha}} v_{\rho_n}\left(\rho_n^{\frac{2}{1-\alpha}}x\right)
            \eeq
            has a subsequence converging in $H^1$ to a positive radially symmetric Choquard ground state $w_0\in \cS_1$ of $E_{\ch}\big|_{\cS_1}$, that is $E_\ch(w_0)=M^\ch_1$.
            \smallskip

            \item[$(ii)$]
            As $\rho\to\infty$, the statement $(ii)$ of Theorem \ref{T11} remains valid.

     \end{itemize}

\end{teo}

The critical quantity $\rho^{**}$ in the previous result plays the role of a critical threshold when we can detect mountain pass geometry for every $\rho>\rho^{**}$ (see Remark \ref{ro**}). Theorem \ref{T12}
does not state anything about the existence of normalised solutions at $\rho=\rho^{**}$.
At the critical value $\rho^{**}$ we prove the following.

\begin{prop}
\label{newProp}
Under the assumptions of Theorem \ref{T12}, for every sequence $\{\rho_n\}$ such that $\rho_n>\rho^{**}$ and $\rho_n\to \rho^{**}$, there exist critical points $u_{\rho^{**}}$ and $v_{\rho^{**}}$ of $F|_{\cS_{\rho^{**},\rad}}$ such that, up to a subsequence,
\begin{itemize}
\item[(i)]
$u_{\rho_n}\to u_{\rho^{**}}$ in $H^1$,
\item[(ii)]
$v_{\rho_n}\to v_{\rho^{**}}$ in $H^1$.
\end{itemize}

\end{prop}

\smallskip
Critical points $u_{\rho^{**}}$ and $v_{\rho^{**}}$ are constructed via a limit process as $\rho_n\to \rho^{**}$ and we do not have a variational characterization for $u_{\rho^{**}}$ and $v_{\rho^{**}}$. We also do not know weather $F(u_{\rho^{**}})=F(v_{\rho^{**}})$. These questions are related to the precise definition of $\rho^{**}$ in \eqref{SE1} (see also Remark \ref{ro**}).

\begin{rem}\label{rem-MPLI-II}
Mountain pass solutions in Theorem \ref{T12} are constructed from two separate branches, solutions of type I for intermediate $\rho>\rho^{**}$, and of type II for large $\rho\gg\rho^*$.
Type I solutions are ``natural'' mountain pass solutions on $\mathcal S_{\rho,\rad}$ that can be constructed as long as $F$ has a (local) minimum on $\mathcal S_{\rho,\rad}$ (that is for the entire range $\rho>\rho^{**}$) following an adaptation of the classical approach in \cite{Je}, see Section \ref{sMPL-I}.
However, it seems to be difficult to control the asymptotic behaviour of solutions of type I as $\rho\to\infty$, see Remark \ref{Type-I-discussion}. Instead, for all sufficiently large $\rho>\rho^*$ we construct mountain pass solutions of type II using decomposition of the Poho\v zaev manifold approach, which goes back to \cite{T92}, see Section \ref{sMPL-II}. Solutions of type II have the required asymptotic behaviour built-in in the construction.
The asymptotic \eqref{asymptMPL-2} and convergence to the Choquard limit profile is established for the type II solutions. We conjecture that type I and type II solutions actually coincide, see Remark \ref{I=II} for the discussion, and proof of Theorem \ref{T12} for precise definitions and technical details.
\end{rem}

 \begin{rem}
	Similar to Figure \ref{F12} structure of the mass--energy relation curve was observed in the study of
	normalised solutions  of quasilinear Schr\"odinger equations \cite{Je15};  and of the local cubic--quintic NLS
	\beq\label{NLS-35}
	-\Delta u+\lambda u+|u|^4u =|u|^2u,\quad u\in \cS_\rho\,{ \quad \lambda\in\R},
	\eeq
	see \cite{Killip,Je21-2,Lewin}. In this case, Lagrange multipliers are apriori bounded from above and the unique positive solution to \eqref{NLS-35} exists if and only if  $\lambda\in(0,3/16)$.  As a consequence, the asymptotic behaviour of the global minimum branch of normalised solutions of the local problem \eqref{NLS-35} as $\rho\to\infty$ (and $\lambda\to 3/16$, cf. \cite[Theorem 4]{Lewin}), is very different from the Thomas--Fermi limit that we observe in the nonlocal GPP problem \eqref{P}.
\end{rem}

\begin{rem}
	\label{r-noninj}
	As was mentioned earlier, every ground state of the fixed frequency$\eqref{Plambda}$ with $\lambda>0$ is a normalised solution of $(N_{\rho_\lambda})$ with $\rho_\lambda:=|u_\lambda|_2>0$. However, for  $\alpha<1$ and small $\lambda>0$, ground states of $\eqref{Plambda}$ are not normalised minimizers of $(N_{\rho_\lambda})$. Indeed, let $\alpha<1$ and let $u_\lambda$ be a ground--state of $\eqref{Plambda}$ described in Theorem \ref{thmA}. Then as $\lambda\to 0$, the rescaling $w_\lambda(x):=\lambda^{-\frac{2+\alpha}{4}}u_\lambda\big(\lambda^{-1/2}x\big)$
	converges in $H^1$ to a positive radially symmetric ground state $w_*\in H^1\cap C^2$ of the Choquard equation \eqref{C1}, up to a subsequence (see \cite[Theorem 2.5]{LZVM} and Lemma \ref{l-App-L2}). This fact, combined with \eqref{N2.2},
	implies that $F(u_\lambda)\simeq \lambda^\frac{1+\alpha}{2}E_\ch(w_*)>0$ and $\rho_\lambda\simeq\lambda^{-\frac{1-\alpha}{4}}|w_*|_2\to\infty$ as $\lambda\to 0$. However, from \eqref{TFasymp} we know that $m_{\rho_\lambda}=\inf_{\cS_{\rho_\lambda}}F\simeq\rho_\lambda^4 m^\tf_1\to-\infty$ as $\lambda\to\infty$, and hence $u_\lambda$ is not a normalised minimizer of $(N_{\rho_\lambda})$. We conjecture that for small $\lambda$ the fixed frequency ground state $u_\lambda$ coincides with a mountain pass solution of type II (and type I?) at $\rho_\lambda$. However this would be difficult to establish since the uniqueness of the fixed frequency ground states of \eqref{Plambda} is entirely open.
	
	Also recall that as $\lambda\to\infty$, the rescaling $z_\lambda(x):=\lambda^{-1/2} u_\lambda(x)$ converges in $L^2$ and $L^4$ to the unique nonnegative radially symmetric compactly supported positive solution $z_*\in C^{0,1/2}$ of the Thomas--Fermi equation \eqref{eqTF1}.
	This implies that $\rho_\lambda\simeq\lambda^{1/2}|z_*|_2\to\infty$ as $\lambda\to \infty$.
	
	Finally, we know that $\rho_\lambda \ge  \bar K_\alpha>0$ by \eqref{e-barrho}. Thus the mapping $\rho_\lambda:\R_+\to\R_+$ for the GPP equation with $\alpha<1$ is neither injective nor surjective.
	
	The same analysis remains valid for $\alpha\ge 1$ and in this case $\rho_\lambda\simeq\lambda^{\frac{\alpha-1}{4}}|w_*|_2$ as $\lambda\to 0$. Similar shapes of the $\rho_\lambda$ curve were observed in \cite[Theorem 3, 4]{Lewin} for a local double--power NLS, however in the local case admissible values of $\lambda$ are bounded above and $\rho_\lambda:(0,\lambda_*)\to\R^+$.
\end{rem}

 \medskip

 It would be interesting to know if some of the results in this work can be extended to the GPP equations with an external potential $V_{\mathrm{ext}}:\R^3\to\R$,
$$
 -\Delta u+\lambda u+V_{\mathrm{ext}}(x)u+|u|^2u =(I_\alpha*|u|^2)u,\quad u\in \cS_\rho,{ \quad \lambda\in\R},
$$
 for a suitable class of external potentials vanishing at infinity.  The main tools used in our arguments such as system \eqref{abc-syst} and scaling considerations should remain asymptotically valid at least in some regimes. The techniques developed in \cite{BMRV,IM,MRV,LM,ZZ} in the local case can be useful here.

\section{Preliminaries}

Here for completeness we recall several well--known results that were frequently used in this paper.

\medskip
\noindent
{\bf Gagliardo--Nirenberg (GN) inequality} \cite[Theorem 12.83]{L}.  For every $s\in(2,6)$ there exists $\bar c_s>0$ such that ,
\beq\label{GN}
|u|_s\le  \bar c_s |\D u|_2^{3\frac{(s-2)}{2s}}|u|_2^{1-3\frac{(s-2)}{2s}}, \qquad \forall u\in H^1.
\eeq

\begin{rem}
\label{R3.1}
By the Sobolev embedding and by interpolation, it is readily seen that $\bar c_s\le c_*^{3\frac{(s-2)}{2s}}$, where $c_*=\frac{1}{\sqrt{3}}\big(\frac{2}{\pi}\big)^{2/3}$.
See (2.22) in \cite{MRV} for an implicit evaluation of $\bar c_s$.
\end{rem}

\medskip
\noindent
{\bf Hardy--Littlewood--Sobolev (HLS) inequality} \cite[Theorem 4.3]{LL}.
For every $\alpha\in (0,3)$, there exists a constant $c_{\alpha}>0$ such that
\beq\label{HLS}
\int_{\R^3}(I_{\alpha}*|u|^2)|v|^2 dx
\le c_{\alpha}|u|_{\frac{12}{3+\alpha}}^{2}|v|_{\frac{12}{3+\alpha}}^{2},
\qquad \forall u,v\in L^\frac{12}{3+\alpha},
\eeq
where $c_\alpha=\frac{\Gamma((3-\alpha)/2)}{\pi^{2\alpha/3}2^{\alpha/3}\Gamma((3+\alpha)/2)}$, and $c_\alpha\to 1$ as $\alpha\to 0$.

\medskip

As a direct consequence of the HLS inequality, by using suitable triangulations, we obtain the following useful convergence properties of the Riesz potentials (cf. also \cite[Lemma 2.4]{MvS}).
\begin{lemma}
    \label{lemmaConvNonLoc}
    Let $\{u_n\}$ be a sequence in $H^1$ such that $u_n\to u$ in $L^{\frac{12}{3+\alpha}}$. Then $\cN(u_n)\to\cN(u)$ and
    \beq\label{conv2}
    \lim_{n\to\infty} \int_{\R^3}(I_{\alpha}*|u_n|^2)u_n\varphi\, dx=\int_{\R^3}(I_{\alpha}*|u|^2)u\varphi\,  dx,\qquad\forall \varphi\in H^1.
    \eeq
\end{lemma}

\medskip
\noindent
{\bf Interpolation inequalities and Energy-Nehari-Poho\v zaev system.}
    The following $L^2$--estimates follow from the GN and HLS inequalities by interpolation,
    \begin{align}\label{stima1}
        |u|_4&\le \bar c \,\rho^{1/4}|\D u|_2^{3/4}, \qquad \forall u\in \cS_\rho,
        \\
    \label{stima-CB}
    \cN(u)
    &\le
    c_\alpha \,\rho^{\frac{4}{3}\alpha}|u|_4^{4-\frac{4}{3}\alpha},\qquad \forall u\in \cS_\rho,
    \\
    \label{stima2}
        \cN(u)
        &\le
        \bar{\bar c}_\alpha \,\rho^{1+\alpha}|\D u|_2^{3-\alpha},\qquad \forall u\in \cS_\rho,
    \end{align}
where $\bar c:=\bar c_4$ and $\bar{\bar c}_\alpha:=c_\alpha\, \bar c^{\left(4-\frac 43\alpha \right)}$.
As a consequence of \eqref{stima1} and \eqref{stima2}, for every $\rho>0$ the energy $F$ is well-defined on $\cS_\rho$.
It is standard to check that the constrained functional $F\big|_{\cS_\rho}$ is of class $\cC^1$ on $\cS_\rho$ and
critical points of $F\big|_{\cS_\rho}$ are weak normalised solutions of \eqref{P},
i.e. if $u\in \cS_\rho$ is a critical point of $F\big|_{\cS_\rho}$ then
$$
\int_{\R^3}\nabla u\cdot\nabla \varphi\, dx+\int_{\R^3}|u|^2 u\varphi \,dx-\int_{\R^3}(I_{\alpha}*|u|^2)u\varphi\, dx
=-\lambda_\rho \int_{\R^3}u\varphi \, dx,\qquad\forall\varphi\in H^1,
$$
for an unknown Lagrange multiplier $\lambda_\rho\in\R$.

In particular, if $u\in \cS_\rho$ is a weak normalised solution of \eqref{P} with a Lagrange multiplier $\lambda_\rho\in\R$ then $u$ satisfies the {\em Nehari identity}
\begin{equation}\label{e-Nehari}
    |\nabla u|_2^2+\lambda_\rho\rho^2+|u|_4^4-\cN(u)=0.
\end{equation}
It is standard to see that under minor regularity assumptions $u$ also satisfies the {\em Poho\v{z}aev identity}.
The proof is an adaptation of \cite[Proposition 3.1]{MvS}.

\begin{prop}[Poho\v{z}aev identity]\label{p-Phz}
    Let $u\in \cS_\rho$ be a weak normalised solution of \eqref{P} with a Lagrange multiplier $\lambda_\rho\in\R$.
    If $\nabla u\in  H^{1}_{loc}$ then
    \begin{equation}\label{e-Phz}
        \frac{1}{2}|\nabla u|_2^2+\frac{3}{2}\lambda_\rho\rho^2 +\frac{3}{4}|u|_4^4-\frac{3+\alpha}{4}\cN(u)=0.
    \end{equation}
\end{prop}

\medskip

For a normalised solution $u\in\cS_\rho$ that satisfies conditions of Proposition \ref{p-Phz}, it is convenient to denote
\begin{equation}\label{abc}
    A=|\nabla u|_2^2,\quad B=|u|_4^4,\quad C=\cN(u).
\end{equation}
Then for any $\alpha\in(0,3)$ we can solve explicitly the Energy--Nehari--Poho\v zaev system
\begin{equation}\label{abc-syst}
    \left\{
    \begin{aligned}
        \frac12A+\frac14 B-\frac1{4} C&=\mu_\rho\\
        A+B-C&=-\lambda_\rho\rho^2\\
        \frac12 A+\frac34B-\frac{3+\alpha}{4}C&=-\frac32\lambda_\rho\rho^2,
    \end{aligned}
    \right.
\end{equation}
to deduce algebraic relations between various quantities involved. For instance, we obtain
\begin{equation}\label{abc-syst-plus}
    \mu_\rho=\frac 14\,\frac{2(1-\alpha) A-\alpha B}{3-\alpha},\qquad\lambda_\rho\rho^2 =\frac{(1+\alpha) A+\alpha B}{3-\alpha},\qquad C=\frac{4A+3B}{3-\alpha}.
\end{equation}
In particular,  the 2nd relation in \eqref{abc-syst-plus} shows that under a mild regularity assumption we can rule out negative and zero Lagrange multipliers.

\begin{prop}\label{abc-rem}
    Let $u\in \cS_\rho$ be a weak normalised solution of \eqref{P} with a Lagrange multiplier $\lambda_\rho\in\R$. If $\nabla u\in  H^{1}_{\loc}$ then $\lambda_\rho>0$.
\end{prop}

Since every normalised solutions of \eqref{P} is a fixed frequency solution of \eqref{Plambda}, the proofs of the following properties follow from the corresponding properties of the weak solutions of \eqref{Plambda}, see  Theorem \ref{thmA} and \cite[Propositions 4.2 and 4.3]{LZVM} (see \cite[p.3092]{MvS-JDE} for the discussion of the asymptotic \eqref{eq-asy}).

\begin{prop}
    \label{reg}
    Let $u\in \cS_\rho$ be a weak normalised solution of \eqref{P}
    with the corresponding Lagrange multiplier $\lambda_\rho> 0$.
    Then $u\in L^1 \cap C^2\cap W^{2,s}$, for every $s>1$.  In addition, if $u\ge 0$ then $u(x)>0$ for all $x\in\R^3$,
    and $u$ satisfies the exponential decay bound \eqref{eq-asy}.
\end{prop}

In particular, if $\lambda_\rho> 0$ then {\em every} weak normalised solution of \eqref{P} has sufficient regularity to ensure that Poho\v zaev identity is valid and system \eqref{abc-syst} holds.

If $\lambda_\rho\le 0$ then at least in principle we may have weak normalised solution of \eqref{P} with $\nabla u\not\in  H^{1}_{\loc}$. However this possibility can be ruled out in some special cases.
For instance, using only the Nehari identity and the energy, we deduce a useful relation
\begin{equation}\label{abc-syst-lambda}
    \lambda_\rho\rho^2 =-4\mu_\rho+A.
\end{equation}
This shows that for every weak normalised solution of \eqref{P} at {\em non-positive energy level} we must have $\lambda_\rho>0$. Hence Poho\v zaev identity is valid and system \eqref{abc-syst} holds.

For non-negative weak normalised solution of \eqref{P} we have the following.

\begin{lemma}\label{l-minus}
    Let $u\in \cS_\rho$ be a weak normalised solution of \eqref{P} with a Lagrange multiplier $\lambda_\rho\in\R$.
    If $u\ge 0$ then $\lambda_\rho\ge 0$.
\end{lemma}

\proof
Since $u\ge 0$ and $u\neq 0$, the following linearised inequality holds
\begin{equation}\label{e-APP}
    -\Delta u +(\lambda_\rho+|u|^2)u\ge 0\quad\text{in $H^1$}.
\end{equation}
Then, according to the Agmon-Allegretto-Piepenbrink  positivity principle \cite[Theorem 3.3]{Agmon},
\begin{equation}\label{e-APP+}
    \int|\nabla\varphi|^2+\int(\lambda_\rho+|u|^2)\varphi^2\ge 0\quad\forall\varphi\in C^\infty_c.
\end{equation}
Assume $\lambda_\rho<0$ and for a given $\varphi\in C^\infty_c\cap \cS_1$ and for $t\ge 0$ consider the family  $\varphi_t(x)=t^{3/2}\varphi(tx)$.
Then as $t\to 0$,
\begin{multline*}
    \int|\nabla\varphi_t|^2+\int(\lambda_\rho+|u|^2)\varphi_t^2=t^2\int|\nabla\varphi|^2+\lambda_\rho\int|\varphi|^2
    +\int|u|^2|\varphi_t|^2\\
    \le \lambda_\rho+t^2|\nabla\varphi|_2^2+t^{3/2}\rho^2|\varphi|_\infty^2\to\lambda_\rho<0,
\end{multline*}
which is a contradiction to \eqref{e-APP+}.
\qed

We do not know if the possibility $\lambda_\rho<0$ can be ruled out for all nodal solutions at positive energy levels, without the additional regularity requirement.

\medskip
\noindent
{\bf A connection between normalised minimizers and fixed frequency ground states.}
Given $\lambda\ge 0$, denote
$$\mathcal N_\lambda:=\{u\in H^1\setminus\{0\}: |\nabla u|_2^2+\lambda|u|^2_2+|u|_4^4-\cN(u)=0\}$$
the Nehari manifold of the $\lambda$--frequency problem $\eqref{Plambda}$.
The next two statements follow closely the arguments in \cite[Proposition 2.3 and Theorem 1.3]{Dovetta} (see also \cite{Je21-1} for similar results).

\begin{lemma}
    \label{ground-min}
    Let $\lambda\ge 0$, $v\in \mathcal N_\lambda$ and $\rho>0$. Then
    \beq\label{eq-lambda-rho}
    F_\lambda(v)\ge m_\rho+\frac12\lambda\rho^2.
    \eeq
    Equality holds in \eqref{eq-lambda-rho} if and only if $v\in \cS_\rho$ and it is both a minimizer for $F_{|\cS_\rho}$ and a $\lambda$--frequency ground state of $\eqref{Plambda}$.
\end{lemma}

\proof
Let $\rho>0$ and $v\in \mathcal N_\lambda$. By the definition of Nehari manifold, for every $t>0$ it holds $F_\lambda(tv)\le F_\lambda(v)$, with strict inequality unless $t=1$. Hence, given any $\rho>0$,
\beq\label{eq-lambda-rho+}
F_\lambda(v)\ge F_\lambda\Big(\frac{\rho}{|v|_2}v\Big)=F\Big(\frac{\rho}{|v|_2}v\Big)+\frac12\lambda\rho^2\ge m_\rho+\frac12\lambda\rho^2,
\eeq
since $\frac{\rho}{|v|_2}v\in\cS_\rho$ and \eqref{eq-lambda-rho} is proved.

To conclude, note that if $v\in S_\rho$ is a minimizer for $F_{|\cS_\rho}$ then
$$F_\lambda(v)= F(v)+\frac12\lambda\rho^2= m_\rho+\frac12\lambda\rho^2,$$
that is equality in \eqref{eq-lambda-rho} holds.

Conversely, if equality in \eqref{eq-lambda-rho} holds for some $v\in \mathcal N_\lambda$ then equality also holds in \eqref{eq-lambda-rho+}. This implies that $|v|_2=\rho$ and hence $v\in \cS_\rho$ and $F(v)=m_\rho$.
Moreover, $v$ must be a minimizer on $\mathcal N_\lambda$ and hence a  $\lambda$--frequency ground state of $\eqref{Plambda}$,
otherwise there exists $w\in \mathcal N_\lambda$, $w\neq v$ and such that $F_\lambda(w)<F_\lambda(v)=m_\rho+\frac12\lambda\rho^2$, which contradicts \eqref{eq-lambda-rho}.
\qed

\smallskip
As a consequence of the last statement of Lemma \ref{ground-min}, it is readily seen that the following corollary holds.

\begin{cor}\label{cor-ground-min}
If $v\in \cS_\rho$ is a minimizer for $F_{|\cS_\rho}$ then $v$ is a $\lambda_v$--frequency ground state of $(P_{\lambda_v})$. Moreover, any other $\lambda_v$--frequency ground state of $(P_{\lambda_v})$ belongs to $\cS_\rho$ and is a minimizer for $F_{|\cS_\rho}$.
\end{cor}

\section{Global minimizers for $\alpha\in(1,3)$ and proof of Theorem \ref{T11}}

\subsection{Global minimizers for all $\rho>0$}
First we are going to establish  the existence of global minimizers for $F|_{\cS_\rho}$ for each $\rho>0$.

\begin{prop}
\label{Pmin}
Let $\alpha\in(1,3)$ and $\rho>0$.
Then $m_\rho<0$ and there exists a positive normalised solution $u_\rho\in\cS_\rho$ of \eqref{P} such that $F(u_\rho)=m_\rho$.
Moreover, $u_\rho$ is radially symmetric and $u_\rho\in L^1 \cap C^2\cap W^{2,s}$, for every $s>1$.
\end{prop}

\proof
In order to show that $m_\rho<0$, let us fix $\bar u\in \cS_\rho$, for $t>0$ consider $t^{3/2}\bar u(tx)\in \cS_\rho$ and observe that $\alpha>1$ gives
$$
\inf_{t\in (0,+\infty)} F(t^{3/2}\bar u(tx))=\inf_{t\in (0,+\infty)} \frac 1 2|\n \bar u|_2^2\,  t^2
+ \frac 14 |\bar u|_4^4\,   t^3-  \frac{1}{4}\cN(\bar u) \, t^{3-\alpha}<0.
$$
To show that $m_\rho$ is finite, we point out that by \eqref{stima-CB}
\beq
\label{1051}
F(u)\ge
\frac 14 |u|_4^4-\frac{c_\alpha}{4}
\rho^{\frac{4}{3}\alpha}|u|_4^{4-\frac{4}{3}\alpha}
:=g_1(|u|_4^4)\qquad\forall u\in \cS_{\rho},
\eeq
then the desired result comes from $\inf_\R g_1>-\infty$.

\medskip

Now, we are going to prove that $m_\rho$ is achieved.
To this aim, let us first remark that
\beq
\label{ES}
m_\rho=\inf\{F(u)\ :\  u\in \cS_{\rho,\rad}\}.
\eeq
Indeed, for every $u\in H^1$, let  $u^*\in H^1_{\rad}$ be its Schwartz spherical rearrangement.  Then
$$
|\nabla u |_2^2\geq |\nabla u^*|_2^2, \ | u|_2^2=| u^*|_2^2, \ | u|_4^4=| u^*|_4^4, \ \cN(u)\le\cN(u^*)
$$
cf. \cite[Chapter 3]{LL}.

Then we can consider a minimizing sequence $\{u_n\}$ for $m_\rho$ in $\cS_{\rho,\rad}$.
 By \eqref{stima2} we infer
\begin{equation}\label{stima1 F basso eq+}
    F(u)\geq \frac 1 2|\n u|_2^2-\frac{\bar{\bar c}_\alpha}{4}\rho^{1+\alpha}|\n u|_2^{3-\alpha}:=g_2(|\n u|_2 ),\qquad\forall u\in \cS_\rho,
\end{equation}
so that by $\alpha>1$ we get that the minimizing sequence verifies $|\n u_n|_2\leq c$, $\forall n\in\N$.
Moreover, $|u_n|_2\equiv \rho$, so $\{u_n\}$ is bounded in $H^1$ and we can assume
that $u_\rho\in H^1_{\rad}$ exists such that, up to a subsequence,
\beq
\label{eq:convergence}
u_n\to u_\rho \qquad\left\{\begin{array}{l}
\text{strongly in }L^{q},\ \forall q\in(2,6)\\
\text{weakly in } H^1,\ L^2\\
\text{a.e. in }\R^3.
\end{array}\right.
\eeq
Then, by Lemma \ref{lemmaConvNonLoc}, $F(u_\rho)\le \liminf_{n\to\infty}F(u_n)=m_\rho$, so  we only need to prove that $u_\rho \in \cS_\rho$.

By the Ekeland's variational principle (\cite[Proposition 5.1]{E}), we can assume that $\{u_n\}$ is a $(PS)$--sequence for $F$ constrained on
$\cS_{\rho,\rad}$, namely
\beq
\label{eq:psminimo}
\int \D u_n\D v\, dx+\int u_n^3 v\, dx -\int (I_\alpha *u_n^2)u_nv\,dx+
\lambda_n\int u_nv\,dx=o(1)\|v\|\,, \quad  \forall n\in\N,\ \forall v\in \cS_{\rho,\rad}
\eeq
for a suitable sequence of Lagrange multipliers $\{\lambda_n\}$.
Then,
\beq
\label{1809}
\lambda_n \rho^2=-4F(u_n)+\int |\D u_n|^2dx+o(1),
\eeq
and taking into account $F(u_n)\to m_\rho<0$ we get, up to a subsequence,
\beq
\label{eq:lambda>0}
\lambda_n\longrightarrow  \lambda>0.
\eeq
By \eqref{eq:convergence},\eqref{eq:lambda>0} and Lemma \ref{lemmaConvNonLoc} the function $u_\rho$ verifies
\beq
\label{eq:minimo}
\int \D  u_\rho \D v\, dx+\int  u_\rho ^3 v\, dx -\int (I_\alpha * u_\rho^2)u_\rho v\,dx+
 \lambda\int_{\R^3}u_\rho v\,dx=0 \quad  \forall v\in \cS_{\rho,\rad}.
\eeq
By the weak convergence, $\liminf_{n\to\infty} |\D u_n|_2\ge |\D u_\rho|_2$ and $\liminf_{n\to\infty} | u_n|_2\ge  |u_\rho|_2$, hence by \eqref{eq:convergence}, \eqref{eq:psminimo},  \eqref{eq:lambda>0} and \eqref{eq:minimo} we conclude that $| u_n|_2 \to |  u_\rho|_2$, and hence $u_\rho\in \cS_\rho$.

Since the functional $F$ is even, we can choose $u_n\ge 0$, $\forall n\in\N$ and conclude that  $u_\rho \ge 0$.
Further, since $\lambda>0$ by \eqref{eq:lambda>0} the regularity results in Proposition \ref{reg} hold and, in particular, $u_\rho>0$.
\qed

\begin{prop}
\label{PF}
Let $\alpha\in(1,3)$ and $\rho>0$.
The map $\rho\mapsto m_\rho$ is continuous, strictly decreasing and strictly concave.
\end{prop}
\proof
The continuity follows from the concavity.

First, we are going to prove the monotonicity.
Let $0<\rho_1<\rho_2$ and observe that $F(u_{\rho_1})<0$ implies
$$
\int_{\R^3} u^4_{\rho_1}dx-\cN(u_{\rho_1}) <0.
$$
Then
\begin{eqnarray*}
\nonumber m_{\rho_2}&\le& F\left(\frac{\rho_1}{\rho_2}\, u_{\rho_1}\right)\\
&=&\frac 12\left(\frac{\rho_2}{\rho_1}\right)^2\int_{\R^3}|\D u_{\rho_1}|^2dx+
\frac 14\left(\frac{\rho_2}{\rho_1}\right)^4\left(\int_{\R^3} u^4_{\rho_1}dx-
\cN(u_{\rho_1})\right)\\
  &<&\left(\frac{\rho_2}{\rho_1}\right)^2 F(u_{\rho_1})\\
\nonumber &<& m_{\rho_1},
\end{eqnarray*}
that proves the claim.

To establish the concavity, let us fix $t\in(0,1)$ and set $\bar \rho=t\rho_1+(1-t)\rho_2$.
By definition of $m_\rho$ we see
\begin{equation*}
\begin{split}
t\, m_{\rho_1}+(1-t)\, m_{\rho_2}\le &\,  t\, F\left(\frac{\rho_1}{\bar\rho}u_{\bar\rho}\right)+(1-t)\, F\left(\frac{\rho_2}{\bar\rho}u_{\bar\rho}\right) \\
=&\left[t\,\left(\frac{\rho_1}{\bar\rho}\right)^2+(1-t)\,\left(\frac{\rho_2}{\bar\rho}\right)^2\right]  \tilde A-
\left[t\,\left(\frac{\rho_1}{\bar\rho}\right)^4+(1-t)\,\left(\frac{\rho_2}{\bar\rho}\right)^4\right]  \tilde D,
\end{split}
\end{equation*}
where $\tilde A=\frac 12 |\D u_{\bar\rho}|_2^2$ and $\tilde D=\frac 14 \left(|u_{\bar\rho}|_4^4-\cN(u_{\bar\rho})\right)$.
Since $m_{\bar\rho}=\tilde A-\tilde D$, then the strict concavity is equivalent to
\beq
\label{reb}
\left[t\,\left(\frac{\rho_1}{\bar\rho}\right)^2+(1-t)\,\left(\frac{\rho_2}{\bar\rho}\right)^2-1\right]  \tilde A-
\left[t\,\left(\frac{\rho_1}{\bar\rho}\right)^4+(1-t)\,\left(\frac{\rho_2}{\bar\rho}\right)^4-1\right]  \tilde D<0.
\eeq
From $m_{\bar\rho}\le 0$ there follows $0<\tilde A\le \tilde D$, so \eqref{reb} holds  because
$$
1<\frac {t\rho_1^2+(1-t)\rho_2^2} {\bar\rho^2} <
\frac {t\rho_1^4+(1-t)\rho_2^4} {\bar\rho^4},
$$
by the convexity of $\tau\mapsto\tau^2$ and Lemma \ref{CVXLemma} below.
\qed

\smallskip

\begin{lemma}
\label{CVXLemma}
Let $\rho_1,\rho_2>0$ and $t\in(0,1)$, then for every $1\le p_1 <p_2$
$$
\frac{t\rho_1^{p_1 }+(1-t)\rho_2^{p_1 }}{(t\rho_1+(1-t)\rho_2)^{p_1}}<\frac{t\rho_1^{p_2}+(1-t)\rho_2^{p_2}}{(t\rho_1+(1-t)\rho_2)^{p_2}}.
$$
\end{lemma}

\proof Let $\gamma> 1$ such that $p_2=p_1\gamma$.
Then by convexity we obtain
$$
\frac{t\rho_1^{p_1\gamma}+(1-t)\rho_2^{p_1\gamma}}{(t\rho_1+(1-t)\rho_2)^{p_1\gamma}}
>\left[\frac{t\rho_1^{p_1 }+(1-t)\rho_2^{p_1 }}{(t\rho_1+(1-t)\rho_2)^{p_1}}\right]^\gamma
>
\frac{t\rho_1^{p_1 }+(1-t)\rho_2^{p_1 }}{(t\rho_1+(1-t)\rho_2)^{p_1}}.\eqno\qed
$$

\medskip

Our next step is to study the asymptotic behaviour of the branch of the global minima $u_\rho$ as $\rho\to 0$ and $\rho\to\infty$.

\begin{prop}
\label{ThAs>1}
Let $\alpha\in(1,3)$ and $\rho>0$. Then
\begin{align}
\label{eAs>1 infty}
m_\rho&\sim -\rho^4\quad \mbox{ as }\rho\to +\infty,\\
\label{eAs>1 0}
m_\rho&\sim -\rho^{2\frac{\alpha+1}{\alpha-1}}\quad\mbox{ as }\rho\to 0.
\end{align}
\end{prop}

\proof
Minimizing $g_1$ in \eqref{1051}, we conclude that
\beq
\label{1933+}
{m_\rho\gtrsim -\rho^4}.
\eeq
To obtain a lower bound, take $w\in \cS_1$ be such that $F(w)<0$. Then $\rho w\in \cS_\rho$ and
\beq
\label{1247}
F(w_\rho)=\left(\frac 12  |\nabla w |^2_2\right)  \rho^2+\frac 14\left( |w|_4^4 -  \cN(w)\right)  \rho^4.
\eeq
Since $\frac 14\left(  |w|_4^4-  \cN(w) \right)<F(w)<0$, then \eqref{1247} gives
$$
m_\rho\lesssim -\rho^4\qquad\mbox{ at }+\infty,
$$
that together with \eqref{1933+} completes the proof of \eqref{eAs>1 infty}.

\smallskip

Minimizing $g_2$ in \eqref{stima1 F basso eq+}, we obtain
\beq
\label{1933++}
{m_\rho\gtrsim -\rho^{2\frac{\alpha+1}{\alpha-1}}}.
\eeq
Next, take the minimizing function $u_1\in \cS_1$ and consider the inversion of the rescaling \eqref{eq-beta1},
\beq
\label{1202}
w_\rho(x):=\rho^{\frac{\alpha+2}{\alpha-1}} u_1\big(\rho^{\frac{2}{\alpha-1}} x\big)\in\cS_\rho.
\eeq
We infer
$$
m_\rho\le  F(w_\rho)=\left(\frac12 |\nabla u_1 |^2_2-\frac14 \cN(u_1) \right)\, \rho^{2\frac{\alpha+1}{\alpha-1}}+ \frac 14 |u_1|_4^4 \, \rho^{2\left(\frac{\alpha+1}{\alpha-1}+\frac{\alpha}{\alpha-1}\right)},
$$
where $\frac12 |\nabla u_1 |^2_2-\frac14 \cN(u_1)\le F(u_1)=m_1<0$.
Then
$$
m_\rho\lesssim -\rho^{2\frac{\alpha+1}{\alpha-1}}\qquad\mbox{as $\rho\to 0$ },
$$
that together with \eqref{1933++} completes the proof of \eqref{eAs>1 0}.
\qed

In order to state also the asymptotic behaviour of the ground state solution, we analyse the  behaviour of the related Lagrange multiplier.
\begin{lemma}
\label{lMin}
Let $\alpha\in(1,3)$ and $\rho>0$.
The Lagrange multiplier $\lambda_\rho>0$ that corresponds to the global minimizer $u_\rho$ satisfies
\begin{align}\label{1132}
\lambda_\rho &\sim \rho^2\quad\mbox{ as }\rho\to\infty,\\
\label{1133}
\lambda_\rho &\sim \rho^{\frac {4}{\alpha-1}}\quad\mbox{ as }\rho\to 0.
\end{align}
\end{lemma}
\proof
From $m_\rho<0$ and \eqref{abc-syst-lambda} it follows that $\lambda_\rho>0$.
Set
$$A_\rho=|\nabla u_\rho|_2^2, \quad B_\rho=|u_\rho|_4^4, \quad C_\rho=\cN(u_\rho).$$
From \eqref{abc-syst} we infer
\begin{align*}
\lambda_{\rho}\rho^2&=\frac{1+\alpha}{3-\alpha}A_\rho+\frac{\alpha}{3-\alpha}B_\rho,\\
-m_\rho&=\frac{\alpha-1}{2(3-\alpha)}A_\rho+\frac{\alpha}{4(3-\alpha)}B_\rho.
\end{align*}
Hence
 $$
\lambda_{\rho}\rho^2\sim-m_\rho\quad\text{ as $\rho\to 0$ and $\rho\to\infty$}.
$$
Then \eqref{1132} and \eqref{1133}  follow from \eqref{eAs>1 infty} and \eqref{eAs>1 0}, respectively.
\qed

\subsection{The Thomas--Fermi limit for $\rho\to \infty$}\label{ssTF}

Let $\rho\to\infty$. Consider the normalized family
\begin{equation*}
    z_\rho(x):=\rho^{-1}u_\rho(x),
\end{equation*}
as defined in \eqref{eq2res}, and recall that
\beq
m_\rho\sim -\rho^4, \qquad \lambda_\rho \sim \rho^2.
\eeq
Then
$|z_\rho|_2^2=1$, and for every $\rho>0$,
$z_\rho$ is a global minimum on $\cS_1$ of the rescaled energy
$$
\tilde{F}_\rho(z)=\frac{\rho^{-2}}{2}|\nabla z|_2^2+\frac{1}{4}|z|_4^4-\frac14\cN(z).
$$
Moreover,
$$
\tilde m_\rho:=\min_{\cS_1}\tilde{F}_\rho=\tilde{F}_\rho(z_\rho)=\frac{m_\rho}{\rho^4}\sim -1\quad\text{as $\rho\to\infty$},
$$
and
$z_\rho$ satisfies the Euler--Lagrange equation
$$
-\rho^{-2}\Delta z_\rho+\lambda_\rho\rho^{-2} z_\rho +z_\rho^3=(I_\alpha*z_\rho^2)z_\rho\quad\text{in $\R^3$},
$$
where $\lambda_\rho\sim \rho^2$.

\begin{prop}\label{pTFconv}
Let $\alpha\in(1,3)$.
Then
$$
\tilde m_{\rho}\to m^\tf_1,\qquad \lambda_{\rho}\rho^{-2}\to-4 m^\tf_1,\quad\mbox{ as }\rho\to\infty,
$$
and the rescaled minimizers
$z_{\rho}$ converge strongly in $L^2$ and $L^4$ to the radially symmetric minimizer
$z_\infty$ of the Thomas--Fermi minimization problem \eqref{eqTF2}, in $ \cS_1$.
\end{prop}

\proof
Taking $z_\rho\in \cS_1$ as a trial function for $E_\tf$, we see immediately that
\beq
\label{TF1}
m^\tf_1\le\tilde m_\rho-\frac{\rho^{-2}}{2}|\nabla z_\rho|_2^2<\tilde m_\rho\qquad \forall \rho>0.
\eeq

Now, let $z_\infty\in C^{0,1/2}\cap \cS_1$ be the unique compactly supported ground state of the Thomas--Fermi problem \eqref{eqTF2} with $\rho=1$,
see Proposition \ref{tTF}.
Note that
$$
z_\infty\le c(R_*-|x|)^{1/2} \quad\text{in $B_{R_*}$,}
$$
where $R_*>0$ is the support radius of
$z_\infty$; and
$z_\infty$ is smooth inside the support.
Yet we can not conclude that
$z_\infty\in H^1$ because of the possible singularity of the gradient on the boundary of the support
(and we know that
$z_\infty\not\in H^1$ when $\alpha=2$, see Remark \ref{rGPP}).

If
$z_\infty\in H^1$, we can use
$z_\infty$ as a trial function for $\tilde F_\rho$ and, by using \eqref{TF1}, we get the estimate
$$
m_1^{\tf}\le \tilde m_\rho\le \tilde F_\rho(z_\infty)=E_\tf(z_\infty)+\frac{\rho^{-2}}{2}|\nabla z_\infty|_2^2=
m^\tf_1+\frac{\rho^{-2}}{2}|\nabla z_\infty|_2^2
$$
that implies $\tilde m_\rho\to m_1^{\tf}$ as $\rho\to\infty$.

Now, assume
$z_\infty\not\in H^1$.
For $r>0$ we introduce the family of cut-off functions $\eta_{r}(x):=\bar\eta \big(r(|x|-R_*)\big)$,
where $\bar\eta\in C_c^{\infty}(\R,[0,1])$ verifies
$\bar\eta(t)= 1$ for $t\le -1$ and  $\bar\eta(t)= 0$ for $t \geq -\frac 12$.
Then, we set
$$
z_r(x):=\tau_r\eta_r(x)z_\infty(x),
$$
where $\tau_r:=|\eta_r z_\infty|_2^{-1}$, so that $|z_r|_2=1$.
As $r\to \infty$, it is elementary to see that $\tau_r\to 1$ and
\begin{equation}\label{eq1205}
|z_r|_4^4=|z_\infty|_4^4+  o(1) ,\quad \cN(z_r)=
    \cN(z_\infty)+  o(1) .
\end{equation}
Set
$$
A_r:=|\nabla z_r|_2^2.
$$
Since  $z_\infty\not\in H^1$, by Fatou Lemma $A_r\to\infty$.
Set $\rho_r=A_r$. Then as $r\to\infty$,
$$
\tilde m_{\rho_r}\le \tilde F_{\rho_r}(z_r)=E_\tf(z_r)+\frac{\rho_r^{-2}}{2}A_r=m^\tf_1+o(1)+\frac{1}{2A_r}\to m^\tf_1.
$$
Therefore, by the continuity of $r\mapsto A_r$ and by \eqref{TF1} we can conclude that, as $\rho\to\infty$,
$$
m_1^{\tf}\le \tilde m_{\rho}\le  m^\tf_1+o(1),
$$
that implies again  $\tilde m_\rho\to m_1^\tf$ as $\rho\to\infty$.

\smallskip

Now, let $\{\rho_n\}$ any sequence such that $\rho_n\to\infty$, then
$$
m^\tf_1\le E_\tf(z_{\rho_n})\le \tilde F_{\rho_n}(z_{\rho_n})=\tilde m_{\rho_n}\to m^\tf_1,
$$
i.e.
$\{z_{\rho_n}\}$ is a minimizing sequence for $m^\tf_1$.
Using an adaptation of the arguments in \cite[Theorem II.1]{Lions-I} we can conclude that
$z_{\rho_n}$ converges to the
$z_\infty$ strongly in $L^2$ and $L^4$, up to a subsequence.
We outline a direct argument for completeness (cf. \cite[Proposition 6.1]{LZVM}).

Recall that the sequence $\{z_{\rho_n}\}$ verifies $|z_{\rho_n}|_2\equiv 1$ and $z_{\rho_n}$ is monotone radially decreasing.
Moreover,  $|z_{\rho_n}|_4$ is bounded, indeed  by \eqref{stima-CB}
$$
E_\tf(z_{\rho_n})\ge \frac 14 |z_{\rho_n}|_4^4-c\, |z_{\rho_n}|_4^{4-\frac 43\alpha}.
$$
Then, by the Strauss's radial lemma \cite[Lemma A.IV]{BeL83}
\begin{equation}\label{rhoStrauss}
0\le    z_{\rho_n}\leq U(|x|):=C\min\{|x|^{-\frac{3}{2}}, |x|^{-\frac{3}{4}}\}\quad\text{for all $|x|>0$}.
\end{equation}
By Helly's selection principle for monotone functions (cf. \cite[Proof of Theorem 3.7]{LL}),
there exists a nonnegative radially symmetric nonincreasing function $\tilde z_\infty(|x|)\le U(|x|)$ such that, up to a subsequence,
$$
z_{\rho_n}(x)\to\tilde z_\infty(x)\quad\text{a.e. in $\R^3$ as $n\to \infty$}.
$$
Moreover, we can assume $z_{\rho_n} \rightharpoonup \tilde z_\infty$ weakly in $L^2$ and in $L^4$.
Since $U\in L^s$ for all $s\in(2,4)$, by the Lebesgue's dominated convergence we conclude that
$z_{\rho_n}\to \tilde z_\infty$ in $L^s$, for every $s\in (2,4)$.
Since  $s=\frac{12}{3+\alpha}\in (2, 4)$, we get $\cN(z_{\rho_n})\to \cN(\tilde z_\infty)$ by Lemma  \ref{lemmaConvNonLoc}.
Observe that $|\tilde z_\infty|_2\le 1$ and that, by the Brezis-Lieb lemma \cite[Theorem 1.9]{LL},
$$
\lim_{n\to\infty}\big(|\tilde z_\infty|_4^4+|z_{\rho_n}-\tilde z_\infty|_4^4\big)=\lim_{n\to\infty}|z_{\rho_n}|_4^4.
$$
Moreover
$$
E_\tf\left(\frac{\tilde z_\infty}{|\tilde z_\infty|_2}\right)\ge m^\tf_1\ \Longrightarrow\ E_\tf(\tilde z_\infty)\ge |\tilde z_\infty|_2^4\, m^\tf_1.
$$
Then
$$
\aligned
4m^\tf_1=
4\lim_{n\to\infty}
E_\tf(z_{\rho_n})&=\lim_{n\to\infty}\big(|\tilde z_\infty|_4^4+|z_{\rho_n}-\tilde z_\infty|_4^4\big)-\lim_{n\to\infty}\cN(z_{\rho_n})\\
&=4E_\tf(\tilde z_\infty)+\lim_{n\to\infty}|z_{\rho_n}-\tilde z_\infty|_4^4\\
&\geq4|\tilde z_\infty|_2^4\,  m^\tf_1+\lim_{n\to\infty}|z_{\rho_n}-\tilde z_\infty|_4^4,
\endaligned
$$
which implies $|\tilde z_\infty|_2=1$ and $|z_{\rho_n}-\tilde z_\infty|_4\to 0$, since $m^\tf_1<0$.
We conclude $\tilde z_\infty\in \cS_1$, $E_\tf(\tilde z_\infty)=m^\tf_1$, and hence
$\tilde z_\infty=z_\infty$ by the uniqueness, and furthermore $z_{\rho_n}\to  z_\infty$ strongly in $L^2$ and $L^4$.

To prove $\lambda_{\rho}\rho^{-2}\to-4 m^\tf_1$, observe that by \eqref{abc-syst}
$$
\lambda_\rho \rho^2 =-\frac 12 \big(|u_\rho|_4^4-\cN(u_\rho)\big)-2m_\rho,
$$
hence
$$
\frac{\lambda_\rho}{\rho^2}=-\frac 12\big(|z_\rho|_4^4-\cN(z_\rho)\big)-2\tilde m_\rho.
$$
Then the claim follows, as $\rho\to\infty$, taking into account $|z_\infty|_4^4-\cN(z_\infty)=4m^\tf_1$.
\qed

\begin{rem}
    Note that all arguments in this subsection remain valid for all $\alpha\in(0,3)$,
    the restriction $\alpha>1$ was only needed in connection with the existence of the normalized solutions $u_\rho$ of \eqref{P} for every $\rho>0$ in Proposition \ref{Pmin}.
\end{rem}

\subsection{The Choquard limit for $\rho\to 0$}

Let $\rho\to 0$ and recall that
\beq\label{lr-CH}
m_\rho\sim -\rho^{2\frac{\alpha+1}{\alpha-1}},\quad\lambda_\rho \sim \rho^{\frac {4}{\alpha-1}}\quad\text{as $\rho\to 0$}.
\eeq
Given $u\in\cS_\rho$, consider the rescaling \eqref{eq-beta1},
\begin{equation}\label{eq1res}
u(x)\quad\mapsto\quad w(x):=\rho^{-\frac{\alpha+2}{\alpha-1}} u\left(\rho^{-\frac{2}{\alpha-1}}x\right)\in \cS_1
\end{equation}
and (cf.~\eqref{Ch-scale1}) denote
$$\bar{F}_\rho(w):=E_\ch(w)+\tfrac{1}{4}\rho^\frac{2\alpha}{\alpha-1}|w|_4^4=\rho^{-2\frac{\alpha+1}{\alpha-1}}F(u).$$
Therefore,
$$\check m_\rho:=\min_{\cS_1}\bar{F}_\rho=\bar{F}_\rho(w_\rho)=\rho^{-2\frac{\alpha+1}{\alpha-1}} m_\rho\sim -1\quad\text{as $\rho\to 0$},$$
where $u_\rho\mapsto w_\rho$ is the rescaling \eqref{eq1res} of the global minimizer of $F$ on $\cS_\rho$.
Moreover, $w_\rho$ satisfies the Euler--Lagrange equation
$$-\Delta w_\rho+\lambda_\rho \rho^{-\frac {4}{\alpha-1}} w_\rho + \rho^\frac{2\alpha}{\alpha-1}w_\rho^3=(I_\alpha*w_\rho^2)w_\rho\quad\text{in $\R^3$},$$
where $\lambda_\rho\sim \rho^{\frac {4}{\alpha-1}}$ by \eqref{lr-CH}.

\begin{prop}\label{pCconv}
    Let $\alpha\in(1,3)$.
    Then \eqref{luc} holds and for any sequence $\rho_n\to 0$, the sequence of rescaled minimizers $\{w_{\rho_n}\}$ has a subsequence that converges strongly in $H^1$ to a positive radially symmetric minimizer $w_0\in \cS_1$ of the Choquard minimization problem \eqref{eqC}. \end{prop}

 \proof
Taking $w_\rho\in \cS_1$ as a trial function for $E_{\ch}\big|_{\cS_1}$, we immediately note that $m^\ch_1<\check m_\rho$.
Let $w_0\in \cS_1$ be a positive minimizer for \eqref{eqC}.
Then
\beq\label{mrhoch}
m_1^\ch<\check m_\rho\le\bar{F}_\rho(w_0)=m^\ch_1+\frac{1}{4}\rho^\frac{2\alpha}{\alpha-1}|w_0|_4^4\to m^\ch_1\quad\text{as $\rho\to 0$}.
\eeq
So, we can conclude  $m_\rho=\rho^{2\frac{\alpha+1}{\alpha-1}}\check m_\rho\simeq \rho^{2\frac{\alpha+1}{\alpha-1}}m^\ch_1$, that proves the first relation in \eqref{luc}.

It follows from \eqref{mrhoch} that for any sequence $\rho_n\to 0$, $\{w_{\rho_n}\}$ is a radially symmetric minimising sequence for \eqref{eqC}.
Note that
$$
m^\ch_\delta:=\inf_{\cS_\delta} E_\ch=\delta^{2\frac{\alpha+1}{\alpha-1}}m^\ch_1,
$$
and since $m^\ch_\delta<0$ and $\alpha>1$,
$$m^\ch_1<m^\ch_\delta+m^\ch_{1-\delta}\qquad\forall\delta\in(0,1).$$
We conclude that $\{w_{\rho_n}\}$ converges strongly in $H^1$ to a positive radially symmetric minimizer $w_0\in \cS_1$ for \eqref{eqC}, up to a subsequence, e.g. by \cite[Theorem III.1]{Lions-I} (minimal adaptation is needed to extend from $\alpha=2$ to $\alpha\in(1,3)$).

Now, by the Nehari identity \eqref{e-Nehari}, the scaling \eqref{eq1res} and the convergence $w_{\rho_n}\to w_0$ in $H^1$, we deduce
\begin{equation*}
\begin{split}
-\lambda_{\rho_n}\rho_n^2&= |\D u_{\rho_n}|_2^2+|u_{\rho_n}|_4^4-\cN(u_{\rho_n})\\
& = \rho_n^{2\frac{\alpha+1}{\alpha-1}}\left(|\D w_0|_2^2-\cN(w_0)+o(1) +\rho_n^{\frac{2\alpha}{\alpha-1}}|w_{\rho_n}|_4^4\right).
\end{split}
\end{equation*}
Hence from \eqref{DR} we infer
$$
\lambda_{\rho_n}=-\left(2\frac{\alpha+1}{\alpha-1}+o(1)\right)m^\ch_1\rho_n^{\frac{4}{\alpha-1}},
$$
that proves the second asymptotic estimate in \eqref{luc}, in view of the arbitrary choice of $\rho_n$.
\qed

\section{Minimum and mountain pass solutions for $\alpha\in(0,1)$ \\ and proof of Theorem \ref{T12}}
\label{S5.1}

\subsection{Existence and nonexistence}
The case $\alpha\in(0,1)$ is quite different from the case $\alpha\in(1,3)$.
A feature of this case is that the energy functional $F$ is coercive on the fibers not only at infinity but also near zero.
Indeed, we will see in Theorem \ref{NE} that now there exists no solution for $\rho$ small.
Nevertheless, we are going to show that for $\rho$ away from zero it is possible to recognize a mountain pass structure and a (local) minimum.

We start with a nonexistence result for small masses $\rho^2>0$.

\begin{teo}
\label{NE}
Let $\alpha\in(0,1)$. Then there exists
\begin{equation}\label{e-barrho}
\bar\rho\, \ge  \bar K_\alpha:= \left(K_\alpha^{\frac 1 2}\, c_\alpha^{\frac{1}{2\alpha} } \, \bar c^{\frac 43}\right)^{-1}
\end{equation}
where $K_\alpha>0$ is defined in \eqref{def C},
such that for every $\rho < \bar\rho$
problem \eqref{P} has no normalised solution $\bar u$  such that    $\D\bar u \in H^1_{\loc}$.
\end{teo}

\begin{rem}
Note that $  \bar K_\alpha \to+\infty$ as $\alpha\to 0$ and $ \to K\in[\sqrt{3}\pi,\infty)$ as $\alpha\to 1$, by Remark \ref{R3.1}, according to the behaviour of the problem in the limit cases.
Moreover,  $\bar K_\alpha$ is monotone decreasing w.r.t. $\alpha\in(0,1)$.
\end{rem}

First we establish the following result.

\begin{lemma}
\label{lNE}
Let $\alpha\in(0,1)$.
There exists $\bar\rho>0$ such that for every $u\in \cS_\rho$, $\rho\in(0,\bar\rho)$ the map
$$
t\mapsto \Psi(t):=F(t^{3/2}u(tx))
$$
has strictly positive derivative.
\end{lemma}

\proof
Let us write
\beq
\label{1817}
\frac{d\,}{dt}\Psi(t)= t \left(|\D u|^2_2+\frac34 |u|_4^4\, t-\frac{3-\alpha}{4}\cN(u)\, t^{1-\alpha} \right).
\eeq
A direct computation shows
\beq
\label{1822}
\min_{t\ge 0}\left(\frac34 |u|_4^4\, t-\frac{3-\alpha}{4}\cN(u)\, t^{1-\alpha} \right)=-\left[\frac{\cN(u)}{|u|_4^{4(1-\alpha)}}\right]^{\frac1\alpha} K_\alpha,
\eeq
where
\begin{equation}\label{def C}
K_\alpha=\frac 3 4\,  \alpha\left(\frac{3-\alpha}{3}\right)^{\frac 1 \alpha}(1-\alpha)^{ \frac{1-\alpha}{\alpha} }.
\end{equation}
By  \eqref{stima1} and \eqref{stima-CB} we get
\beq
\label{1823}
\left[\frac{\cN(u)}{|u|_4^{4(1-\alpha)}}\right]^{\frac1\alpha}=
\left[\frac{\cN(u)}{|u|_4^{4 -\frac 43 \alpha}}\right]^{\frac1\alpha}|u|_4^{\frac 83}\le
c_\alpha^{\frac 1 \alpha}\bar c^{\frac 83}\rho^2|\D u|^2_2.
\eeq
Combining the estimates \eqref{1822} and \eqref{1823} in \eqref{1817} we obtain
\begin{equation}\label{derivata psi}
\frac{d\,}{dt}\Psi(t)\ge t \left[1-K_\alpha c_\alpha^{\frac 1 \alpha}\bar c ^{\frac 83}\,\rho^2\right]|\D u|_2^2
\end{equation}
and for
\begin{equation}\label{stima rho}
\rho<
\left(K_\alpha^{\frac 1 2} c_\alpha^{\frac{1}{2\alpha} }\bar c^{\frac 43}\right)^{-1}
\end{equation}
the right-hand side in \eqref{derivata psi} remains positive.
\qed

\begin{proof}[Proof of Theorem \ref{NE}]
Let $\bar u$ be a solution of problem \eqref{P} such that  $\D\bar u \in H^1_{\loc}$.
Then system \eqref{abc-syst} provides
$$
 \frac{d\,}{dt}\Psi(t)_{|_{t=1}}=
 |\D u|^2_2+\frac34 |u|_4^4 -\frac{3-\alpha}{4}\cN(u)=0,
 $$
that implies $\rho\ge \bar \rho$ by Lemma \ref{lNE}.
\end{proof}

A key point for the existence proofs is the following lemma.

\begin{lemma}
\label{<0}
Let $\alpha\in(0,1)$ and set
\beq
\label{1907}
\rho^*=\inf_{\rho>0}\{m_\rho<0\}.
\eeq
Then $\rho^*\in(0,\infty)$.

\end{lemma}

\proof
 {\em  $\rho^*<\infty$}:\quad
In order to show that $m_\rho<0$ for $\rho$ large, we first claim that there exists $\bar u\in H^1_{\rad}$ such that
\beq
\label{1118}
\int_{\R^3}\bar u^4\, dx -\cN(\bar u) < 0.
\eeq
Indeed, let us fix $\phi\in \cC^{\infty}_0$ and for $t>0$ let us compute
\beq
\label{1119}
\int_{\R^3}\bar \phi^4(tx)\, dx=\frac{1}{t^3} \int_{\R^3}\bar \phi^4(x)\, dx,
\qquad
\cN(\phi(tx))=\frac{1}{t^{3+\alpha}}\cN(\phi ).
\eeq
Putting \eqref{1119} in \eqref{1118}, we see that there exists a small $\bar t$ such that $\bar u(x):=\phi(\bar t x)$ verifies the claim.

Now, from
$$
F(\rho \bar u)=\frac 12 \left( \int_{\R^3} |\D \bar u|^2dx\right)\, \rho^2 +\frac 14\left(\int_{\R^3}\bar u^4dx-\cN(\bar u)\right)\, \rho^4
$$
it follows that $F(\rho \bar u) < 0$ for large $\rho$, so $\rho^*<\infty$ holds.

\bigskip

{\em $\rho^*>0$}:\quad By \eqref{stima-CB} and \eqref{stima2},
\beq
\label{956}
F(u)\ge\frac 1 2|\n u|_2^2+
\frac 14 |u|_4^4-\frac{c_\alpha}{4}
\rho^{\frac{4}{3}\alpha}|u|_4^{4-\frac{4}{3}\alpha}
\qquad\forall u\in \cS_\rho,
\eeq
and
$$
F(u)\geq \frac 1 2|\n u|_2^2-\frac{\bar{\bar c}_\alpha}{4}\rho^{1+\alpha}|\n u|_2^{3-\alpha}\qquad\forall u\in \cS_\rho.
$$
Assume that $\bar{\bar c}_\alpha \rho^{1+\alpha}<1$. Since $\alpha<1$, for all $u\in \cS_\rho$ such that $|\n u|_2\le 1$ we have
$$
F(u)\geq \frac 1 4|\n u|_2^2,\qquad\forall u\in \cS_\rho.
$$
Next for $u\in \cS_\rho$ such that $|\n u|_2\ge 1$ we estimate
$$
F(u)\ge
\frac12+\frac 14 |u|_4^4-\frac{c_\alpha}{4}
\rho^{\frac{4}{3}\alpha}|u|_4^{4-\frac{4}{3}\alpha}\ge \frac12-c_1\rho^4 > 0
\qquad\forall u\in \cS_\rho
$$
for $\rho^4 < \frac{1}{2c_1}$ and hence $\rho_*>0$.
\qed

\medskip

Let us now identify the topological structure that ensures the existence of a mountain pass and of a local minimum.
For $\alpha\in(0,1)$, inequality \eqref{stima1 F basso eq+} implies
\beq
\label{513-1}
\cB_\rho:=\inf\{F(u)\ :\ u\in \cS_{\rho,\rad},\ |\D u|_2=R_\rho\}\ge g_2(R_\rho)>0,
\eeq
where $R_\rho$ is the maximum point of the function $g_2$ in \eqref{stima1 F basso eq+}, namely
\beq
\label{714}
R_\rho=c\rho^{-\frac{1+\alpha}{1-\alpha}},\qquad c>0.
\eeq
Let us set
$$
\overline m_\rho=\inf\{F(u)\ :\ u\in \cS_{\rho,\rad},\ |\D u|_2> R_\rho\}
$$
and
\beq
\label{SE1}
\rho^{**}=\inf\{\rho>0\ :\  \cB_\rho>\overline m_\rho\}.
\eeq

\begin{rem}\label{stima rho** basso}
By \eqref{stima rho} and $|\D (t^{3/2}u(tx))|_2=t|\D u|_2$, $\forall t>0$,  it immediately follows that
\begin{equation}\label{stima rho** basso eq}
\rho^{**}\ge\left( K_\alpha^{\frac 1 2}c_\alpha^{\frac{1}{2\alpha} }\bar c^{\frac 43}\right)^{-1}.
\end{equation}
\end{rem}

\begin{prop}\label{P4.4}
    Let $\alpha\in(0,1)$. Then $\overline m_\rho=m_\rho$ for all $\rho\ge\rho^*$, and
$$0<\rho^{**}<\rho^*.$$
\end{prop}
\proof
Equality $\overline  m_\rho=m_\rho$, $\forall \rho\ge\rho^*$, is a direct consequence of the definition of $m_\rho$, $R_\rho$ and $\overline m_\rho$, taking into account  \eqref{ES}.
Inequality $\rho^{**}<\rho^*$ follows from the continuity of the maps $\rho\mapsto g_2(R_\rho)$
and  $\rho\mapsto F(\rho \bar u)$, for every $\bar u\in H^1\setminus\{0\}$.
Inequality $\rho^{**}>0$ comes from \eqref{stima rho** basso eq}
\qed

\medskip

We are going to show that a local minimizer $\overline m_\rho$ exists for all $\rho>\rho^{**}$.
Since  $\overline m_\rho=m_\rho$ for all $\rho\ge\rho^*$, in this case the solution is a global minimum.

\begin{teo}
\label{Tmin}
Let $\alpha\in(0,1)$ and $\rho>\rho^{**}$.
Then problem \eqref{P} has a normalised solution $u_\rho\in \cS_\rho$, that is a local minimum for
$F_{|_{\cS_{\rho,\rad}}}$.
Moreover,
the corresponding Lagrange multiplier $\lambda_\rho$ is strictly positive,
\beq
\label{reb2}
F(u_\rho)=\overline m_\rho,
\eeq
$u_\rho \in L^1 \cap C^2\cap W^{2,s}$ for all $s>1$ and it is positive and radially symmetric.
\end{teo}

\proof
We only need to verify the existence of $u_\rho$ that satisfies \eqref{reb2}.
Indeed, $u_\rho$ turns out to be a free critical point for $F$ on $ \cS_{\rho,\rad}$  by \eqref{SE1} and it is a critical point for $F$ on $\cS_\rho$ by the  principle of symmetric criticality \cite{Pal}.
Hence it is a solution of  \eqref{P} and it can be proved to be positive in the same way as as in Proposition \ref{Pmin}.

Let $\{u_n\}$ in $\cS_{\rho,\rad}$ be a minimizing sequence for $\overline m_\rho$, with $|\D u_n|_2 > R_\rho$ and $F(u_n)<\cB_{\rho}$.
The sequence $\{u_n\}$ is bounded in $H^1$ because by \eqref{956} there exists a constant $c(\rho)>0$ such that
\beq
\label{Sinn}
\cB_{\rho}>F(u_n)\ge \frac12\int_{\R^3}|\D u_n|^2dx -c(\rho)\qquad \forall n\in\N.
\eeq

We claim that $u_n$ can be assumed to verify the  identity
\beq
\label{ES3}
|\D u_n|_2^2+\frac{3}{4}|u_n|_4^4-\frac{3-\alpha}{4}\cN(u_n)=0.
\eeq
Indeed, again by  \eqref{956},  we can choose $u_n$ in such a way that
$$
F(u_n)=\min\{F(t^{3/2}(u_n(tx))\ :\ t\in(t_1,t_2)\}
$$
where $0<t_1<1<t_2$ are such that $|\D t_1^{3/2}(u_n(t_1 x))|_2=R_\rho$  and $F(t_2^{3/2}(u_n(t_2x))>\cB_\rho$ (the value $t_2$ exists by \eqref{Sinn}).
Then \eqref{ES3} follows from $\frac{d\,}{dt}F(t^{3/2}u_n(tx))_{|_{t=1}}=0$.

Finally, by the Ekeland principle  we obtain a bounded minimizing sequence $\{v_n\}$ such that
$$
\|v_n-u_n\|=o(1),\qquad F'(v_n)_{|_{\cS_{\rho,\rad}}}=o(1).
$$
Then,  we obtain an approximation of a system similar  to   \eqref{abc-syst}, with \eqref{ES3} in place of the Poho\v{z}aev identity \eqref{e-Phz},
\begin{equation}\label{abc-syst_n}
    \left\{
    \begin{aligned}
        \frac12A_n+\frac14 B_n-\frac1{4} C_n&=m_\rho+o(1)\\
        A_n+B_n-C_n&=-\lambda_n\rho^2+o(1)\\
        A_n+\frac{3}{4}B_n-\frac{3-\alpha}{4}C_n& =o(1)
    \end{aligned}
    \right.
\end{equation}
for a suitable sequence $\{\lambda_n\}$ in $\R$, where
$$
A_n:=|\nabla v_n|_2^2, \quad B_n:=|v_n|_4^4, \quad C_n:= \cN(v_n).
$$
By the boundedness of $\{v_n\}$ in $H^1$ and $|\D  v_n(x)|_2>R_\rho+o(1)$, up to a subsequence we have
$$
A_n\to A\ge R_\rho^2>0,\qquad
B_n\to B \ge 0,\qquad
C_n\to C \ge 0,\qquad
\lambda_n\to\lambda_\rho\in\R.
$$
Going to the limit in \eqref{abc-syst_n} we obtain
$$
\lambda_\rho=\frac{1}{\rho^2}\left(\frac{1+\alpha}{3-\alpha}A+\frac{\alpha}{3-\alpha}B\right)>0.
$$
Now, up to a subsequence, assume that $v_n\to u_\rho$ as in \eqref{eq:convergence}.
Then, arguing as in the proof of Proposition \ref{Pmin} we conclude that $u_\rho$ is a local minimum  point  for $F$ constrained on $ \cS_{\rho,\rad}$ and $u_\rho>0$ has the desired regularity properties.
\qed

\medskip

\begin{rem}
From the nonexistence result in Theorem \ref{NE} and the existence result in Theorem \ref{Tmin} we can again deduce $\rho^*>0$ in Lemma \ref{<0}.
\end{rem}

The following proposition describes several properties of the map $\overline m_\rho$.

\begin{prop}
\label{PCM}
Let $\alpha\in(0,1)$ and $\rho>\rho^{**}$.
The map $\rho\mapsto \overline m_\rho$ is  strictly decreasing and continuous.
Moreover, for $\rho\ge \rho^*$ it is strictly concave.
\end{prop}

\proof
To verify the monotonicity, let us fix $\rho^{**}<\rho_1<\rho_2$ and let $u_{\rho_1}$ be a minimizing function as in \eqref{reb2}.
Since $u_{\rho_1}$ is a local minimum, as in \eqref{ES3}  we get
$$
|\D u_{\rho_1}|_2^2+\frac 34 |u_{\rho_1}|_4^4- \frac{3-\alpha}{4}\cN(u_{\rho_1})
=0,
$$
so that
$$
\frac 34 \left(|u_{\rho_1}|_4^4-\cN(u_{\rho_1}) \right)
=-\left(|\D u_{\rho_1}|_2^2+\frac{\alpha}{4}\cN(u_{\rho_1})\right)<0.
$$
Now, let us set $u=\frac{\rho_2}{\rho_1}u_{\rho_1}\in \cS_{\rho_2}$, then
$$
|\D u|_2>|\D u_{\rho_1}|_2>R_{\rho_1}>R_{\rho_2}
$$
(see \eqref{714}) and
\begin{eqnarray}
\overline m_{\rho_2}\le F(u)&=&\frac 12\left(\frac{\rho_2}{\rho_1}\right)^2\int_{\R^3}|\D u_{\rho_1}|^2dx+
\frac 14\left(\frac{\rho_2}{\rho_1}\right)^4\left(\int_{\R^3} u^4_{\rho_1}dx-\cN(u_{\rho_1})\right)
\nonumber\\
&<&\left(\frac{\rho_2}{\rho_1}\right)^2 F(u_{\rho_1})
\label{nf}
\\
&\le & \overline m_{\rho_1}.
\nonumber
\end{eqnarray}
In the last line the equal sign holds when $F(u_{\rho_1})=0$.

\smallskip

To prove the continuity, let $\bar\rho>\rho^{**}$ and $\{\rho_n\}$ in $(\rho^{**},\infty)$ any sequence such that $\rho_n\to\bar\rho$.
Then $\left|\D \left(\frac{\rho_n}{\bar\rho}u_{\bar\rho}\right)\right|_2>R_{\rho_n}$ for large $n$, by continuity, and
$$
\overline m_{\rho_n}\le F\left(\frac{\rho_n}{\bar\rho}u_{\bar\rho}\right)=(1+o(1))F(u_{\bar\rho})=(1+o(1))\overline m_{\bar\rho},
$$
so that
\beq
\label{941}
\limsup_{n\to\infty}\overline  m_{\rho_n}\le\overline  m_{\bar\rho}.
\eeq
On the other side, $\left|\D \left(\frac{\bar\rho}{\rho_n}u_{\rho_n}\right)\right|_2>R_{\rho_n}$, for large $n$, and
$$
\overline m_{\bar \rho}\le F\left(\frac{\bar \rho}{\rho_n}u_{\rho_n}\right)=(1+o(1))F(u_{\rho_n})=(1+o(1))\overline m_{\rho_n},
$$
so that
\beq
\label{942}
\overline  m_{\bar\rho}\le \liminf_{n\to\infty}\overline m_{\rho_n}.
\eeq
From \eqref{941} and \eqref{942} the continuity of the map $\rho\mapsto\overline  m_\rho$ follows.

\smallskip

The concavity for $\rho\ge \rho^*$ can be proved as in Proposition \ref{PF}, taking the advantage of $\overline m_{\rho}\le 0$, $\forall \rho\ge\rho^*$.
\qed

\begin{cor}
Let $\alpha\in (0,1)$. Then
\begin{itemize}
\item[(i)]
$\overline m_\rho=0$ if and only if $\rho=\rho^*$,
\item[(ii)]
$m_\rho=0$ for all $\rho\in(0,m^*)$.
\end{itemize}

\end{cor}
\proof
{\em (i)} is a direct consequence of the definition of $\rho^*$ in \eqref{1907} and the strict monotonicity stated in Proposition \ref{PCM}.

{\em (ii)} follows from the definition of $\rho^*$ and from
\beq
\label{1905}
m_\rho\le 0,\qquad
\forall\rho>0.
\eeq
To verify \eqref{1905}, let $\bar u\in \cS_\rho$, then $t^{3/2}\bar u(tx)\in \cS_\rho$,   $\forall t>0$, and $\lim\limits_{t\to 0}F(t^{3/2}\bar u(tx))=0$.
\qed

\subsection{A Mountain Pass solution of type I}\label{sMPL-I}
We are going to show that for all $\rho>\rho^{**}$ the problem \eqref{P} admits a normalised mountain pass solution $v_\rho$ at a positive energy level, that we will call {\sl solutions of type I} . These type I solutions are ``natural'' mountain pass solutions on $\mathcal S_\rho$ that can be constructed as long as $F$ has a (local) minimum on $\mathcal S_\rho$ (that is for the entire range $\rho>\rho^{**}$) following an adaptation of the classical approach in \cite{Je}.

\begin{teo}
\label{TMP}
Let $\alpha\in(0,1)$ and $\rho>\rho^{**}$. Then problem \eqref{P} has a normalised solution $v_\rho$, that is a mountain pass critical point for $F_{|_{\cS_{\rho,\rad}}}$.
Moreover, the corresponding Lagrange multiplier $\widetilde\lambda_\rho>0$,
\beq
\label{525+1}
F(v_\rho)\ge\cB_{\rho},
\eeq
and $v_\rho \in L^1 \cap C^2\cap W^{2,s}$ for all $s>1$, and $v_\rho$ is positive.
\end{teo}

\proof
Let  $\rho>\rho^{**}$ and define
\beq\label{def}
\begin{split}
\Gamma=\bigg\{\gamma:[0,1]\to \cS_{\rho,\rad}\ :\, &\ |\D \gamma(0)|_2<R_\rho,\, |\D \gamma(1)|_2>R_\rho, \\
& \ F(  \gamma(0))< \frac{g_2(R_\rho)}{2},\, F(\gamma(1))<\frac{\overline m_\rho+ g_2(R_\rho)}{2}\bigg\},
\end{split}
\eeq
where $R_\rho$ is from \eqref{714}.
Observe that, by taking the local minimum $u_\rho$ provided by Theorem \ref{Tmin} and considering
$$
\tau\mapsto u_{\rho,\tau}(x):=\tau^{3/2} u_\rho (\tau x)\in \cS_{\rho,\rad},\qquad \tau>0,
$$
we conclude that $\Gamma\neq\emptyset$. Indeed, by direct computations,
$$
\lim_{\tau \to 0^+}|\D u_{\rho,\tau}|_2=0,\qquad \lim_{\tau\to 0^+}F(u_{\rho,\tau})=0.
$$
Hence, for $\rho>\rho^{**}$ we can define the mountain pass value by
\beq \label{MPv}
\widetilde M_{\rho}:=\inf_{\gamma\in\Gamma}\max_{s\in [0,1]}F(\gamma(s)).
\eeq

Now, arguing as in \cite[Proposition 3.11]{BMRV} (see also \cite{Je}) we can find a bounded Palais-Smale sequence $\{u_n\}$  at the level $\widetilde M_\rho$ that almost verifies \eqref{ES3}, namely  $\{u_n\}$ satisfies system \eqref{abc-syst_n} with $\widetilde M_\rho$ in place of $m_\rho$, and moreover
\beq
\label{quasipos}
\lim_{n\to\infty}|u_n^-|_2=0.
\eeq
Since $\widetilde M_\rho\ge\cB_{\rho}>0$ and, by \eqref{stima1} and \eqref{stima2}, $F_{|_{\cS_\rho}}(u)\to 0$ as $|\D u|_2\to 0$,  we see that $|\D u_n|_2\ge c>0$.
This point ensures that passing to the limit in \eqref{abc-syst_n} we obtain a nontrivial system, that is equivalent to \eqref{abc-syst}. Then by the 2nd relation in \eqref{abc-syst-plus} we conclude that there exists a limit Lagrange multiplier $\widetilde \lambda_\rho>0$.

As a consequence, we can argue as in the proof of Proposition  \ref{Pmin} and  show that $u_n$ strongly converges to a critical point $v_\rho$ at the level $\widetilde M_\rho$.
By \eqref{quasipos}, $v_\rho\ge0$.
Since $\widetilde\lambda_\rho>0$ we can apply Proposition \ref{reg}, and the maximum principle, so the proof is complete.
\qed

\begin{prop}
\label{PCMP}
Let $\alpha\in(0,1)$ and $\rho>\rho^{**}$.
The map $\rho\mapsto \widetilde M_\rho$ is  continuous.
\end{prop}

\proof
Let $\rho_n,\bar\rho\in(\rho^{**},\infty)$ be such that $\rho_n\to\bar \rho$ with $\rho_n<\bar\rho$.
Set $\Gamma_n,\Gamma$ according to \eqref{def}.
For large $n$, to every $\bar\gamma_n\in\Gamma_n$ corresponds a curve $t\mapsto\gamma_n(t):=\frac{\bar \rho}{\rho_n}\bar\gamma_n(t)\in\Gamma$, by the continuity of the map $\rho\to\overline m_\rho$.
Since $\rho_n<\bar \rho$,  a straightforward estimate shows
\beq
\label{wcd}
\widetilde M_{\bar\rho}\le \max_{\gamma_n}F\le \left(\frac{\bar\rho}{\rho_n}\right)^4\max_{\gamma_n}F.
\eeq
By \eqref{wcd}, and a diagonal argument we obtain
\beq
\label{lis}
\widetilde M_{\bar\rho}\le \liminf_{\rho\to\bar\rho^-} \widetilde M_\rho.
\eeq
On the other hand, for large $n$, to every $\gamma\in\Gamma$ corresponds a curve $t\mapsto\tilde \gamma_n(t):=\frac{ \rho_n}{\bar \rho}\gamma(t)\in\Gamma_n$.
Moreover,
\beq
\label{wcd1}
\widetilde M_{\rho_n}\le \max_{\tilde\gamma_n}F\le \left(\frac{\rho_n}{\rho}\right)^4\max_{\tilde\gamma}F+ \left[\left(\frac{\rho_n}{\rho}\right)^2- \left(\frac{\rho_n}{\rho}\right)^4\right]\max_{t\in[0,1]}|\D\tilde\gamma(t)|_2^2.
\eeq
By \eqref{wcd1},  we obtain
\beq
\label{lis1}
\limsup_{\rho\to\bar\rho^-}\widetilde M_\rho\le \widetilde M_{\overline\rho}.
\eeq
Hence it follows form \eqref{lis} and \eqref{lis1} that
$$
\lim_{\rho\to\bar\rho^-}\widetilde M_\rho= \widetilde M_{\overline\rho}.
$$
A similar argument shows that
\[
\lim_{\rho\to\bar\rho^+}\widetilde M_\rho= \widetilde M_{\overline\rho}.\qedhere
\]

At this point we are ready to give a proof of Proposition \ref{newProp} that shows the existence of normalised solutions at the critical value $\rho^{**}$.

\begin{proof}[Proof of Proposition \ref{newProp}]

{(i)}\quad
We claim that
\beq
\label{f1}
F(u_{\rho_n})<c_1\quad\forall n\in\N,\quad F(u_{\rho_n})>c_2>0 \ \mbox{ for large }n\in\N.
\eeq
The first inequality in \eqref{f1} can be deduced by $F(u_{\rho_n})=\bar m_{\rho_n} <\cB_{\rho_n}$ (see \eqref{SE1}) and $\cB_{\rho_n}\le c_1$, for a suitable constant $c_1>0$, because $\rho_n>\rho^{**}\ge\bar\rho>0$ (see \eqref{513-1} and \eqref{714}).

The second estimate in \eqref{f1} follows because $\rho^{**}<\rho^*$ and $\rho\mapsto \bar m_\rho$ is strictly decreasing.

Taking into account \eqref{f1}, together with the continuity and  coercivity of $F$, we deduce the existence of $c_4,c_3>0$ such that
$$
0<c_3\le |\D u_{\rho_n}|_2\le c_4\quad\mbox{ for large }n\in\N.
$$
Then, we can proceed as in the proof of Proposition \ref{Pmin}.
The only difference is that to deduce the positivity of the limit Lagrange multiplier $\lambda_{\rho^{**}}$, in place of \eqref{1809} we have to consider
$$
\lambda_{\rho_n}\rho_n^2=\frac{1+\alpha}{3-\alpha} |\D u_{\rho_n}|_2^2+\frac{\alpha}{3-\alpha}|u_{\rho_n}|^4_4
$$
(see \eqref{abc-syst-plus}).

\smallskip

{ (ii)}\quad Arguing as before, it is sufficient to prove that $F(v_{\rho_n})$ is bounded, and bounded away from zero.
To prove that $F(v_{\rho_n})$ is bounded we observe that
$$
F(v_{\rho_n})=\widetilde M_{\rho_n}\le\max_{t\in (0,1)}F(t^\frac 32 u_{\rho_n}(tx))\le c
$$
by \eqref{MPv} and by (i) above.
On the other hand, by \eqref{525+1} and \eqref{513-1},
$$
F(v_{\rho_n})\ge g_2(R_{\rho_n})\ge c>0\quad\forall n\in \N,
$$
which completes the proof.
\end{proof}

\begin{rem}
\label{ro**}
A natural question that arises is whether it is possible to define $\rho^{**}$ directly as the infimum of the mass for which a mountain pass structure exists.
Namely, for every $\rho, R>0$, set
\beq
\label{cov1}
\cB_{\rho,R}=\inf\{F(u) : u\in \cS_{\rho,\rad}, |\D u|_2=R\},\quad \cL_{\rho,R}=\inf\{F(u) : u\in \cS_{\rho,\rad}, |\D u|_2\ge R\},
\eeq
and define
$$
\widehat\rho^{**}=\inf\{\rho>0\ :\ \bm{\exists}R>0\text{ such that } \cB_{\rho,R}>\cL_{\rho,R}\}.
$$
We wonder whether it is true that
\beq
\label{cov}
\widehat\rho^{**}=\inf\{\bar\rho>0\ :\ \bm{\forall} \rho>\bar\rho \ \exists R_\rho >0\text{ such that } \cB_{\rho,R_\rho}>\cL_{\rho,R_\rho}\}.
\eeq
\end{rem}
\smallskip

\subsection{Energy and Lagrange multiplier estimates}
Next we are going to establish several estimates which will be used in the study of the asymptotics of the branches of solutions.

\begin{prop}
\label{PropAs<1}
Let $\alpha\in(0,1)$.
Then
\begin{align}
\label{eAs<1 infty}
m_\rho&\sim-\rho^4\quad \mbox{ as }\rho\to \infty,\\
\label{eAs<1 0}
\widetilde M_\rho&\sim \rho^{2-\frac{4}{1-\alpha}}\quad\mbox{ as }\rho\to \infty.
\end{align}
\end{prop}

\proof
Since $m_\rho<0$ for large $\rho$, the proof of  \eqref{eAs<1 infty} proceeds exactly as in the proof of  \eqref{eAs>1 infty}.

\smallskip

From the defintion of $\widetilde M_\rho$ and $\Gamma$ in \eqref{MPv} and \eqref{def}, and from \eqref{stima1 F basso eq+}, we infer
\beq
\label{chiama}
\widetilde M_\rho\ge g_2(R_\rho)\sim \rho^{2-\frac{4}{1-\alpha}}.
\eeq

\smallskip

In order to prove the reverse inequality, take $\bar\rho>\rho^*$ and $\bar u\in \cS_{\bar \rho}$ such that $F(\bar u)<0$.
Then, consider in $\cS_{\bar \rho}$ the family
\beq
\label{com}
w_{t,\rho}(x)=\left(\frac{\bar \rho}{\rho}\right)^{\frac{2+\alpha}{1-\alpha}} \bar u_t\Bigg(\left(\frac{\bar\rho}{ \rho}\right)^\frac{2}{1-\alpha} x\Bigg),
\eeq
where $\bar u_t(y)=t^\frac 32\bar u(ty)$.

Observe that
$$
F(w_{1,\bar\rho})=F(\bar u)<0
$$
and that a computation similar to \eqref{nf} implies that
$$
F(w_{1,\rho})<0\qquad\forall\rho\ge\bar\rho.
$$
Then, for every $\rho\ge\bar\rho$,
\begin{eqnarray}
\nonumber
\widetilde M_\rho&\le&\max_{t\in(0,1)}F(w_{t,\rho})=\max_{t\in (0,1)}\left[\frac{B}{4}\left(\frac{\rho}{\bar \rho}\right)^{4-\frac{6}{1-\alpha}}t^3+
    \Big(\frac{A}{2}t^2-\frac{C}{4}t^{3-\alpha}\Big)\left(\frac{\rho}{\bar \rho}\right)^{2-\frac{4}{1-\alpha}}\right]\\
\label{fTF}
    &\le&
    \max_{t\in (0,1)}\Big(\frac{A}{2}t^2+\frac{B}{4}t^3-\frac{C}{4}t^{3-\alpha}\Big)\left(\frac{\rho}{\bar \rho}\right)^{2-\frac{4}{1-\alpha}}
    \\
    \nonumber
    &=& \left[ \left(\frac{1}{\bar \rho}\right)^{2-\frac{4}{1-\alpha}}  \underbrace{\mathrm{max}_{t\in(0,1)}F(w_{t,\bar\rho})}_{>0}\right]\, \rho^{2-\frac{4}{1-\alpha}},
\end{eqnarray}
where
$$
A:=|\D \bar u|_2^2, \qquad B:=|\bar u|_4^4,\qquad C:=\cN(\bar u).
$$
Hence, \eqref{eAs<1 0} follows from \eqref{chiama} and \eqref{fTF}.
 \qed

\begin{lemma}
\label{lMin0}
Let $\alpha\in(0,1)$ and let $\lambda_\rho$ be the   Lagrange multiplier associated to the minimal solution $u_\rho$.
Then
\beq\label{1112}
\lambda_\rho\gtrsim   \rho^2 \qquad\mbox{ as }\rho\to\infty.
\eeq
\end{lemma}

\proof
By using the usual notations, we see by  \eqref{abc-syst-plus} that
\begin{equation}\label{eq-s-m}
m_\rho=\frac{1-\alpha}{2(3-\alpha)}A_\rho-\frac{\alpha}{4(3-\alpha)}B_\rho
\end{equation}
and
\begin{equation}\label{eq-s-lambda}
\lambda_\rho\rho^2=\frac{1+\alpha}{3-\alpha}A_\rho+\frac{\alpha}{3-\alpha}B_\rho.
\end{equation}
From $m_\rho\sim-\rho^4\to-\infty$ as $\rho\to\infty$ and  \eqref{eq-s-m} there follows,
\beq
\label{1139}
B_\rho\gtrsim\rho^4\qquad(\rho\to\infty),
\eeq
hence \eqref{1112} comes from \eqref{eq-s-lambda}.
\qed
\begin{rem}\label{rem-516}
Using \eqref{1112}, similarly to the proof of Proposition \ref{pTFconv}, we can actually show that $\lambda_\rho \rho^{-2}\to -4 m_1^{\tf}$ as $\rho\to\infty$.
\end{rem}

\begin{rem}\label{Type-I-discussion}
As a next step we would like to provide an asymptotic estimate for the Lagrange multiplier $\widetilde \lambda_\rho>0$ associated to the mountain pass solution $v_\rho$. If we consider the functional $F$ on the fiber
\beq
\label{ephi}
F(t^{3/2}v_\rho(tx))=\frac{A_\rho}{2}t^2+\frac{B_\rho}{4}t^3-\frac{C_\rho}{4}t^{3-\alpha}=:\varphi_{v_\rho}(t),\qquad t>0,
\eeq
then it is natural to expect the mountain pass value to occur at the local maximum of $\vi_{v_\rho}$, that is $t=1$ to be a local maximum of $\varphi_{v_\rho}$. If this were true we would be able to establish the following upper estimate on $\widetilde\lambda_\rho$.

\begin{claim}
    \label{lMP}
    Let $\alpha\in(0,1)$. For every $\rho>\rho^{**}$, let $\widetilde \lambda_\rho>0$ be the Lagrange multiplier associated to the type I mountain pass solution $v_\rho$ at the energy level $\widetilde M_\rho$.
    If for all sufficiently large $\rho$, $t=1$ is a local maximum of the function $\varphi_{v_\rho}$ in \eqref{ephi}, then
    \beq\label{1113}
    \widetilde\lambda_\rho\lesssim \rho^{-\frac{4}{1-\alpha}}.
    \eeq
\end{claim}

\proof
Since $v_\rho$ solves the system \eqref{abc-syst}, by a direct computation we find
$$
\varphi_{v_\rho}''(1)=(1-\alpha)\, \frac{\widetilde\lambda_\rho\rho^2}{2}-(7-\alpha)\, \widetilde M_\rho.
$$
Hence, $\varphi_{v_\rho}''(1)\le 0$ provides
$$
\widetilde\lambda_\rho\le\const\frac {\widetilde M_\rho}{\rho^2}
$$
and the statement follows from \eqref{eAs<1 0}.
\qed

\medskip
However, we were not able to verify that $1$ is a local maximum of $\varphi_{v_\rho}$. Indeed, the mountain pass level $\widetilde M_\rho$ is obtained as a minimax on {\em all} the paths introduced in \eqref{def}.

An estimate of type \eqref{1113} is essential for the asymptotic analysis of the mountain pass solution $v_\rho$.
Since we could not obtain \eqref{1113} by verifying for example that $t=1$ is a local maximum of $\varphi_{v_\rho}$, in the next subsection for all sufficiently large $\rho$ we construct a suitable mountain pass type solution that occurs at the local maximum of its own fiber a priori (see also Remark \ref{Sol=}).
\end{rem}

\subsection{A Mountain Pass solution of type II}\label{sMPL-II}
\label{S5.3}

Here, we use the ideas developed in \cite{T92} to construct a mountain pass {\sl solution ot type II} that is located in the ``negative'' part of the Pohozaev set we are working on. That method was recently adapted to the mass constrained Gross-Pitaevskii-Poisson equation, see \cite{Y_et_Al_arXiv} and further references therein. The information about the location of solutions of type II with respect to the Poho\v zaev set will ensure that the constructed solution occur at the local maximum of its own fiber, which ensures this is not a local minimum, i.e. a mountain pass solutions relative $\cS_\rho$, and will also allow to complete the asymptotic analysis.

\medskip
Following the ideas in \cite{T92} and \cite{Y_et_Al_arXiv} (see also \cite{CJ_SIAM_19}), and using the function $\vi_u$ defined in \eqref{ephi}, we set
\beq
\label{P-}
\P^-_\rho=\left\{u\in \cS_{\rho,\rad}\ :\ \vi'_u(1)=0,\ \vi_u''(1)<0,\ |\D u|_2^{3-\alpha}<H \left(\frac{1}{\rho}\right)^{1-\alpha}\cD(u)\right\},
\eeq
for a fixed constant $H>0$ such that
\beq
\label{eH}
H <\frac{3-\alpha}{4}\,\left(\frac{2}{3\bar c^4}\right)^{1-\alpha}  \alpha^\alpha(1-\alpha)^{1-\alpha},
\eeq
here $\bar c$ is the Gagliardo-Nirenberg constant  \eqref{stima1}.

We are going to study the minimization problem
\beq
\label{M2}
\bar M_\rho=\inf_{\P^-_\rho}F.
\eeq

\begin{teo}
\label{TMP2}
 Let $\alpha\in(0,1)$.
 There exists $\overline\rho^{**}>\rho^*$ such that for every $\rho>\overline\rho^{**}$ the problem \eqref{P} has a normalised radial solution $\bar v_\rho\in \P^-_\rho$ that is a minimizer for $\bar M_\rho$, and with a Lagrange multiplier $\bar\lambda_\rho>0$.
 Moreover,
 \beq
 \label{eAs<1 0++}
 \bar M_\rho\sim \rho^{2-\frac{4}{1-\alpha}}\quad\text{and}\quad \bar\lambda_\rho\lesssim \rho^{-\frac{4}{1-\alpha}}\quad\mbox{as }\rho\to +\infty,
 \eeq
 and $\bar v_\rho \in L^1 \cap C^2\cap W^{2,s}$ for all $s>1$, and $\bar v_\rho$ is positive.
\end{teo}

To prove Theorem \ref{TMP2} we verify in Lemma \ref{L_Fibre} that for large $\rho$, the set $\P^-_\rho$ is not empty and we analyse the fibers corresponding to its points. Then in Lemma \ref{L5.17} we study the asymptotic behaviour of  $\bar M_\rho$ as $\rho\to\infty$.
Finally in Lemma \ref{L5.19} we show that minimizing sequences that approach the infimum are located in the interior of $\P^-_\rho$ and that $\P^-_\rho$ is complete.

With these tools in hands, one can proceed in a usual way: the minimization problem \eqref{M2} can be solved and $\P^-_\rho$ turns out to be a smooth natural constraint, namely critical points of $F$ on it are critical points on $\cS_\rho$. We point out that the key tool for this last statement is an abstract deformation result \cite[Theorem 3.2]{Ghou}. The same abstract result is used in Proposition 3.11 of \cite{BMRV}, to which we refer in Theorem \ref{TMP} to prove that $M_\rho$ is a critical value and that almost \eqref{ES3} and \eqref{quasipos} hold.

We will omit some details of the proof, see \cite[\S 4.1]{Y_et_Al_arXiv} and references therein. However we present all crucial estimates on the geometry of the functional, both for the sake of completeness and to describe our asymptotic analysis. The crucial upper bound on $\bar\lambda_\rho$ in \eqref{eAs<1 0++}, that was missing in the construction of type I mountain pass solutions is established in Proposition \ref{finalmente}.

\begin{lemma}
\label{L_Fibre}
Let $\alpha\in(0,1)$.
Define
$$
U_\rho=\left\{u\in \cS_{\rho,\rad}\;:\ |\D u|_2^{3-\alpha}<H \left(\frac{1}{\rho}\right)^{1-\alpha}\cD(u)\right\}.
$$
Then there exists $\bar\rho>0$ such that for all $\rho>\bar\rho$
\begin{itemize}
\item[$(a)$]
$U_\rho\neq\emptyset${\em ;}
\item[$(b)$]
$u\in U_\rho $ $\iff$ $u_t=t^{3/2}u(t\, \cdot)\in U_\rho$ $\forall t>0${\em ;}
\item[$(c)$]
for every $u\in U_\rho$ there exist a unique $\bar t_u>0$ such that $u_{\bar t_u}\in \P^-_\rho$.
Moreover
\beq
\label{ec}
\left(\frac{4}{3-\alpha}\,\frac {|\D u|_2^2}{\cD(u)}\right)^{\frac{1}{1-\alpha}} <  \bar t_u < \left(\frac{1}{\alpha} \right)^{\frac{1}{1-\alpha}}  \left(\frac{4}{3-\alpha}\, \frac {|\D u|_2^2}{\cD(u)}\right)^{\frac{1}{1-\alpha}} ;
\eeq
\item[$(d)$]
$\min_{t\in(0,+\infty)}F(u_t)<0$ for every $u\in U_\rho$.
\end{itemize}
\end{lemma}
\proof
$(a)$\quad Let us fix $u\in \cS_{1,\rad}$ and let
$$
(u)_\rho(x):= \rho^{-\frac{\alpha+2}{1-\alpha}}u(\rho^{-\frac{2}{
1-\alpha}}x).
$$
Taking into account
$$
|\D (u)_\rho|_2=\rho^{-\frac{\alpha+1}{1-\alpha}}|\D u|_2,\qquad
\cD((u)_\rho)= \rho^{-2\frac{\alpha+1}{1-\alpha}}\cD(u),
$$
we see that $(u)_\rho\in U_\rho$ whenever
$$
\rho^{-2\alpha}<H\frac{\cD(u)}{|\D u|_2^{3-\alpha}},
$$
that is true for all large $\rho$.

\smallskip

$(b)$\quad This statement follows by a direct computation.

\smallskip

$(c)$ \quad As usual, denote $A=|\D u|_2^2$, $B=|u|_4^4$, $C=\cD(u)$.
Taking into account that the derivative of the function $\vi_u$ can be written as
$$
\vi'_u(t)=t^2\left[A\,\left(\frac {1}{t}\right)-\frac{3-\alpha}{4}\,C\, \left(\frac{1}{t}\right)^\alpha+\frac 34\, B\right],
$$
an elementary analysis of $\vi_u$ ensures that it has a local maximum and a minimum, because by \eqref{eH}
\beq
\label{emin}
\min_{t\in(0,+\infty)} \left[A\,\left(\frac {1}{t}\right)-\frac{3-\alpha}{4}\,C\, \left(\frac{1}{t}\right)^\alpha \right]<-\frac 34\, B.
\eeq
Moreover $\vi'_u\left( \alpha^{\frac{1}{1-\alpha}} \bar t  \right)=\left( \alpha^{\frac{1}{1-\alpha}} \bar t  \right)^2\frac 34\, B>0$, where  $\bar t:=\left(\frac{4}{\alpha(3-\alpha)}\, \frac AC\right)^{\frac{1}{1-\alpha}}$  realizes the minimum in \eqref{emin}.
Hence \eqref{ec} holds and $(c)$ follows once we verify
\beq
\label{eempty}
\cA:=\{u\in U_\rho\ :\ \vi'_u(1)=0,\ \vi''_u(1)=0\}=\emptyset.
\eeq

Indeed, on one hand, if $u\in\cA$ then a direct computation and \eqref{stima1}  give
\beq
\label{1456}
 (1-\alpha)A=\frac {3}{4}\, \alpha\,  B\le \frac {3 }{4 }   \alpha\,  \bar c^4\rho A^\frac 32\ \Rightarrow\
 A\ge\left(\frac{4 }{3 }\,\frac{ 1-\alpha }{ \alpha}\,\frac{1}{\bar c^4}\right)^2\left(\frac{1}{\rho}\right)^2.
\eeq
On the other hand, if $u\in\cA$ then, taking into account the definition of $U_\rho$, we infer
\beq
\label{1457}
A=\frac{\alpha(3-\alpha)}{4}\,C> \frac{\alpha(3-\alpha)}{4}\, \frac{\rho^{1-\alpha}}{H}\, A^{\frac{3-\alpha}{2}}
\ \Rightarrow\
A<\left(\frac{4}{\alpha(3-\alpha)}\, H\right)^{\frac{2}{1-\alpha}}\left(\frac{1}{\rho}\right)^2.
\eeq
Inequalities in \eqref{1456} and \eqref{1457} lead to a contradiction because of \eqref{eH}, so the proof of $(c)$ is completed.

\smallskip

$(d)$\quad Arguing as in $(c)$, one can prove that
$$
\min_{t\in(0,+\infty)} \vi_u(t)=\min_{t\in(0,+\infty)} t^3\left[\frac{A}{2}\,\left(\frac1t\right)-\frac{C}{4}\,\left(\frac{1}{t}\right)^\alpha+\frac{B}{4}\right]<0,
$$
because
$$
\min_{t\in(0,+\infty)}\left[\frac{A}{2}\,\left(\frac1t\right)-\frac{C}{4}\,\left(\frac{1}{t}\right)^\alpha\right]<-\frac{B}{4} 
$$
by the definition of $U_\rho$,  \eqref{eH} and  \eqref{stima1}.

\qed

\begin{lemma}
\label{L5.17}
Let $\alpha\in(0,1)$. Then $\bar M_\rho\sim \rho^{2-\frac{4}{1-\alpha}}$ as $\rho\to +\infty$.
\end{lemma}

\proof
Taking into account $(d)$ in Lemma \ref{L_Fibre}, we obtain $\bar M_\rho\ge c\rho^{2-\frac{4}{1-\alpha}}$ as in \eqref{chiama}.

\smallskip

For the reverse inequality, let $\bar\rho>0$ be as in Lemma \ref{L_Fibre}, so that $U_{\bar\rho}\neq\emptyset$ and $\bar u\in U_{\bar\rho}$, and define $w_{t,\rho}$ as in \eqref{com}.
 Then  a direct computation  shows $w_{t,\rho}\in U_\rho$ for all $\rho>\bar\rho$,  and we can proceed as in \eqref{fTF}.
\qed

\smallskip

The following key lemma establishes that the minimizing sequences for $\bar M_\rho$ do not approach the boundary of $\P^-_\rho$.

\begin{lemma}\label{L5.19}
Let $\alpha\in(0,1)$ and $\rho>\overline\rho$.
\begin{itemize}
\item[$(a)$]
If $w_{\rho,n}$ in $U_\rho$ verify $w_{\rho,n}\to w_\rho\in\partial U_\rho$, then, using the notations of $(c)$ in Lemma \ref{L_Fibre},
$$
\lim_{\rho\to\infty}\liminf_{n\to\infty}\frac{E((w_{\rho,n})_{\bar t_{w_{\rho,n}}})}{\bar M_\rho}= +\infty;
$$

\item[$(b)$]
$
\{u\in U_\rho\ :\ \vi'_u(1)=0,\ \vi''_u(1)=0\}=\emptyset.
$
\end{itemize}
\end{lemma}

\proof
$(a)$ Let us fix a large $\rho>\overline\rho$.
Since $w_{\rho,n}\to w_\rho$, by \eqref{ec} we can assume, up to a subsequence, that $\bar t_{w_{\rho,n}}\to t_{\rho}\in (0,+\infty)$ as $n\to\infty$.
Moreover, by continuity, $\vi'_{(w_{\rho})_{t_\rho}}(1)=0$.
Then, using the usual notations for  $w_{t_{\rho}}$,
$$
A_\rho=\frac{3-\alpha}{4}C_\rho-\frac34B_\rho<\frac{3-\alpha}{4}C_\rho=\frac{3-\alpha}{4H}\rho^{1-\alpha}A_\rho^{\frac{3-\alpha}{2}},
$$
where the last equality holds because $w_\rho\in\partial U_\rho$.
Then
\beq
\label{e1}
A_\rho>\left(\frac{4H}{3-\alpha}\right)^{\frac{2}{1-\alpha}}\left(\frac1\rho\right)^2.
\eeq
By $\vi_{(w_{\rho})_{t_\rho}}''(1)\le 0$,
\beq
\label{e2}
-C_\rho\ge -\frac{4}{\alpha(3-\alpha)}A_\rho
\eeq
Hence, by \eqref{e1}, \eqref{e2} and system \eqref{abc-syst},
$$
E((w_\rho)_{ t_{\rho}})=\frac 16 A_\rho-\frac{\alpha}{12}\,C_\rho\ge \frac 16\left(\frac{1-\alpha}{3-\alpha}\right)A_\rho> c\left(\frac1\rho\right)^2.
$$
So $(a)$ follows by Lemma \ref{L5.17}.

\smallskip

$(b)$ has been proved in \eqref{eempty}.
\qed

\medskip

Using Lemmas \ref{L_Fibre}, \ref{L5.19} and \ref{L5.17}, we can find a large $\overline\rho^{**}> \overline\rho$ such that for every $\rho>\overline\rho^{**}$ the minimizer $\bar v_\rho$ belongs to the interior of $\P^-_\rho$ relative to $\cS_{\rho,\rad}$, and hence is a solution of \eqref{P}
with a Lagrange multiplier $\bar\lambda_\rho>0$. In particular, $\bar v_\rho$ satisfies in a standard way regularity and positivity properties stated in Theorem \ref{TMP2}. The next proposition states that $\bar\lambda_\rho$ satisfies the estimate we are looking for in order to capture the asymptotic behaviour of $\bar\lambda_\rho$ as $\rho\to\infty$.

\begin{prop}
\label{finalmente}
Let $\alpha\in(0,1)$ and $\rho>\overline\rho^{**}$.
Let $\bar\lambda_\rho$ be the Lagrange multiplier associated to the solution $\bar
v_\rho$.
Then
$$
0<\bar\lambda_\rho\lesssim \rho^{-\frac{4}{1-\alpha}}.
$$
\end{prop}

\proof Observe that $\bar v_\rho$ satisfies system \eqref{abc-syst}, hence $\bar \lambda_\rho>0$ follows from \eqref{abc-syst-plus}.

For the proof of the asymptotic estimate, we can argue as in Claim \ref{lMP}, taking into account Lemma \ref{L5.17}.
\qed

\smallskip

\begin{rem}\label{I=II}
We conjecture that the mountain pass solutions of type II provided by Theorem \ref{TMP2} actually coincide with the mountain pass solution of type I constructed in Theorem \ref{TMP} when both solutions exist. However, by \eqref{SE1}, \eqref{1907} and $(d)$ in Lemma \ref{L_Fibre} we see that $\rho^{**}<\rho^*<\overline\rho^{**}$, that is the range $(\rho^{**},+\infty)$ where type I solutions exist is larger than the interval $(\bar\rho^{**},+\infty)$  where type II solutions were constructed. Note that for $\rho>\overline\rho^{**}$  it is clear that $\widetilde M_\rho\le \bar M_\rho$, but a priori $\widetilde M_\rho<\bar M_\rho$ could occur.
\end{rem}

\subsection{The Choquard limit for the mountain pass type solutions}
Let $\rho>\overline\rho^{**}$ and $\bar M_\rho$
be the critical level as in \eqref{M2}. Recall that by Lemma \ref{L5.17}
\beq
\label{eAs<1 0+}
\bar M_\rho\sim \rho^{2-\frac{4}{1-\alpha}}\qquad\mbox{ as }\rho\to \infty.
\eeq
As before, let $\bar v_\rho$ be the mountain pass critical point at level $\bar M_\rho$ constructed in Theorem \ref{TMP2}.
Consider the rescaled family as in \eqref{eq1res},
\begin{equation}\label{eq1res+}
     \bar v_\rho(x)\quad\mapsto\quad \bar w_\rho(x):=\rho^{\frac{\alpha+2}{1-\alpha}} \bar v_\rho\left(\rho^{\frac{2}{1-\alpha}}x\right).
\end{equation}
Then
\beq
\label{etrasf}
|\bar w_\rho|_2^2=1,\quad |\nabla\bar  w_\rho|_2^2=\rho^{2\frac{\alpha+1}{1-\alpha}}|\nabla v_\rho|_2^2,\quad |\bar w_\rho|_4^4=\rho^{\frac{4\alpha+2}{1-\alpha}}|v_\rho|_4^4,\quad
\cN(\bar w_\rho)=\rho^{2\frac{\alpha+1}{1-\alpha}}\cN(v_\rho).
\eeq
and for every $\rho>\overline\rho^{**}$,
\beq
\label{Fbar}
\bar{F}_\rho(w):=\frac{1}{2}|\nabla w|_2^2+\frac{1}{4}\rho^{-\frac{2\alpha}{1-\alpha}}|w|_4^4-\frac14\cN(w)=\rho^{2\frac{\alpha+1}{1-\alpha}}F(u),
\eeq
where $u\mapsto w$ is as in  \eqref{eq1res+}.

Let $\widehat M_\rho:=\bar F_\rho(\bar w_\rho)$.
From the scaling \eqref{eq1res+} and \eqref{etrasf}, proceeding as in Section \ref{S5.3} and working with the functions
$$
t\mapsto\bar\vi_v(t):=\bar F_\rho(v_t),\quad v\in \cS_1,\ t\in(0,\infty),
$$
it is readily seen that for $\rho>\overline\rho^{**}$,
\beq
\label{DMbar}
\widehat M_\rho=\min_{\bar \P^-_\rho} \bar F_\rho = \rho^{2\frac{1+\alpha}{1-\alpha}}\bar M_\rho,
\eeq
where
\beq
\label{P-w}
\bar \P^-_\rho=\left\{w\in \cS_{1}\ :\ \vi'_w(1)=0,\ \vi_w''(1)<0,\ |\nabla w|_2^{3-\alpha}<H \rho^{2\alpha}\cD(w)\right\}.
\eeq
Notice that $\bar\vi_v$ depends on $\rho$.
Moreover, $\bar w_\rho$ satisfies the Euler-Lagrange equation
\beq
\label{1619}
-\Delta\bar  w_\rho+\bar \lambda_\rho \rho^{\frac {4}{1-\alpha}} \bar w_\rho + \rho^{-\frac{2\alpha}{1-\alpha}}\bar w_\rho^3=(I_\alpha*\bar w_\rho^2)\bar w_\rho\quad\text{in $\R^3$},
\eeq
where $\bar \lambda_\rho\lesssim \rho^{-\frac{4}{1-\alpha}}$ by Proposition \ref{finalmente}.

\begin{lemma}
\label{todo}
Let $\widehat M_\rho$ be as in \eqref{DMbar}. Then
 $$
 \lim_{\rho\to\infty}\widehat M_\rho=M_1^\ch.
 $$
 \end{lemma}

\proof
First, recall that in the Choquard case the mountain pass value $M_1^\ch$ can be obtained working on the fibers $t\mapsto v_t$ and the corresponding functions
$$
t\mapsto\bar\vi^\ch_v(t):=\bar E_\ch(v_t),\quad v\in \cS_1,\ t\in(0,\infty).
$$

\smallskip

Let $v_0$ be the mountain pass solution of the Choquard equation (see Proposition \ref{tCmin}).
Arguing as in Lemma \ref{L_Fibre} we see that there exists  $t(\rho)$ such that $v_{0,t(\rho)}\in \bar \P^-_\rho$, for  large  $\rho$. In particular, we can characterize
$$
(\vi^\ch_{v_0})'(t)>0 \; \forall t\in(0,1), \quad (\vi^\ch_{v_0})'(1)=0,\quad (\vi^\ch_{v_0})'(t)<0\; \forall t>1,
$$
$$
\bar\vi'_{v_0}(t)>0 \; \forall t\in(0,t(\rho)), \quad \bar\vi_{v_0}'(t(\rho))=0 ,\quad    \bar\vi'_{v_0}(t)<0\;  \forall t\in(t(\rho),t(\rho)+\e(\rho))
$$
with $\e(\rho)>0$,
where
\beq
\label{defvi}
\vi^\ch_{v_0}(t)=\left(\frac 12 |\D v_0|_2^2\right)  t^2-\left(\frac 14\cD(v_0)\right)  t^{3-\alpha},\quad \bar\vi_{v_0}(t)=\vi^\ch_{v_0}(t)+\left(\frac 14 |v_0|_4^4\right) \rho^{-\frac{2\alpha}{1-\alpha}}t^3.
\eeq
We infer $t(\rho)\to 1$ as $\rho\to\infty$, because
$$
\bar\vi'_{v_0}(t)\to (\vi^\ch_{v_0})'(t)>0 \quad \forall t\in(0,1), \quad \bar\vi'_{v_0}(t)\to (\vi_{v_0}^\ch)'(t)<0 \quad \forall t>1.
$$
Then
$$
\widehat M_\rho\le \bar F(v_{0,t(\rho)})=\bar\vi_{v_0}(t(\rho))\longrightarrow  \vi_{v_0}^\ch(1)=M_1^\ch
$$
and it follows that
\beq
\label{elimsup}
\limsup_{\rho\to\infty}\widehat M_\rho\le M_1^\ch.
\eeq

\smallskip

To show the reverse inequality, let $\bar w_\rho$ be as in \eqref{eq1res+} and let  $\tau(\rho)>0$ be such that $(\vi_{\bar w_\rho}^\ch)'(\tau(\rho))=0$.
From $ \bar\vi'_{\bar w_\rho}(1)=0$ we infer $(\vi_{\bar w_\rho}^\ch)'(1)<0$, hence $\tau(\rho)<1$.
Then
$$
M_1^\ch\le \vi_{\bar w_\rho}^\ch(\tau(\rho))\le \bar\vi_{\bar w_\rho}(\tau(\rho))<\bar\vi_{\bar w_\rho}(1)=\widehat M_\rho,
$$
that implies
\beq
\label{eliminf}
M_1^\ch\le\liminf_{\rho\to\infty}\widehat M_\rho.
\eeq
Then, the assertion follows from \eqref{elimsup} and \eqref{eliminf}.
\qed

\begin{prop}\label{Pconv0}
    Let $\alpha\in(0,1)$.
    Then, for any sequence $\rho_n\to \infty$, the sequence of rescaled solutions $\{w_{\rho_n}\}$ converges strongly in $H^1$ up to a subsequence to a positive radially symmetric mountain pass solution of the Choquard  problem \eqref{N2.2}. Moreover, as $\rho\to\infty$,
    $$
    \bar\lambda_\rho \simeq  \lambda_\ch \rho^{-\frac{4}{1-\alpha}}.
    $$
\end{prop}
\proof
By Proposition \ref{tCmin} and \eqref{N2.2}, \eqref{DR}, \eqref{N2.3},
\beq
\label{1723}
S_\alpha=\frac{|\D v_0|_2^{3-\alpha}}{\cD(v_0)}=
\frac{\left(2\,\frac{3-\alpha}{1-\alpha}\, M_1^\ch\right)^{\frac{3-\alpha}{2}}}
{\frac{8}{1-\alpha}M_1^\ch}
=
2^{-\frac{3+\alpha}{2}}\left[\frac{(3-\alpha)^{3-\alpha}}{(1-\alpha)^{1-\alpha}}(M_1^{\ch})^{1-\alpha}\right]^{\frac 12}.
\eeq
Testing $S_\alpha$ by  $\bar w_\rho$, and by using system \eqref{abc-syst} and  \eqref{etrasf},
$$
 S_\alpha\le \frac{|\D \bar w_\rho|_2^{3-\alpha}}{\cD(\bar w_\rho)}=
 \frac
 {\left(2\frac{3-\alpha}{1-\alpha}\widehat M_\rho+\frac12\frac{\alpha}{1-\alpha}\rho^{2\frac{1+\alpha}{1-\alpha}}|\bar v_\rho|_4^4\right)^{\frac{3-\alpha}{2}}}
{\frac{8}{1-\alpha}\widehat M_\rho+\frac{1}{1-\alpha}\rho^{2\frac{1+\alpha}{1-\alpha}}|\bar v_\rho|_4^4}.
$$
We claim that  $\bar w_\rho$ is a minimizing family for $S_\alpha$, as $\rho\to\infty$.
By \eqref{1723} and Lemma \ref{todo}, the claim follows once we prove
\beq
\label{1724}
\lim_{\rho\to\infty}\rho^{2\frac{1+\alpha}{1-\alpha}}|\bar v_\rho|_4^4=0.
\eeq
To verify \eqref{1724}, first observe that by \eqref{abc-syst-lambda}, Proposition \ref{finalmente} and Lemma \ref{L5.17}
\beq
\label{fin2}
|\D \bar v_\rho|_2^2\sim \bar M_\rho\sim\rho^{2-\frac{4}{1-\alpha}}.
\eeq
Then by \eqref{stima1} we obtain
$$
\rho^{2\frac{1+\alpha}{1-\alpha}}|\bar v_\rho|_4^4\lesssim \rho^{-\frac{2\alpha}{1-\alpha}}\to 0\quad\mbox{ as }\rho\to \infty.
$$
Now, let $\rho_n$ be a sequence such that $\rho_n\to\infty$.
Since $\bar w_{\rho_n}$ is a minimizing sequence for $S_\alpha$, then we can conclude by  \cite[Lemma 2.3]{Ye}.

 Moreover, the strong limit of $\bar w_{\rho_n}$ is a solution of \eqref{Ch-GN}, and as in Remark \ref{r-GN} we realise that \eqref{Ch-GN} is equivalent to \eqref{Ch-01}, which implies a-posteriori that $\bar\lambda_\rho\sim   \lambda_\ch \rho^{-\frac{4}{1-\alpha}}$ by \eqref{1619}.
\qed

\begin{rem}
\label{Sol=}
For the mountain pass solutions $v_\rho$ of type I we cannot obtain the asymptotic description provided by Proposition \ref{Pconv0}. Indeed, a basic tool in its proof is the estimate \eqref{fin2}, that follows from Proposition \ref{finalmente}. However we do not have a corresponding result for the Lagrange multiplier $\widetilde\lambda_\rho$ of $v_\rho$. On the other hand, it is not difficult to prove that  $\lim_{\rho\to\infty}  \rho^{2\, \frac{1+\alpha}{1-\alpha}} \widetilde M_\rho = M_1^\ch$, as stated in Lemma \ref{todo} for the mountain pass level of type II.
This supports the conjecture $v_\rho=\bar v_\rho$, for $\rho>\overline\rho^{**}$.
\end{rem}

\subsection{The Thomas--Fermi limit for the minimum solutions}
We observe that all the arguments outlined in the Subsection \ref{ssTF} and in the proof of Proposition \ref{pTFconv} remain valid for $\alpha\in(0,1)$.

\medskip

\section{Critical case $\alpha=1$ and proof of Theorem \ref{T13}}

Recall that the case $\alpha=1$ is $L^2$--critical for the Choquard energy.

\subsection{Existence and nonexistence}

\begin{lemma}
\label{L6.1}
    Let $\alpha=1$.
    Then $m_\rho=0$ for all $\rho\in(0,\rho_*]$ and $-\infty<m_\rho<0$ for all $\rho>\rho_*$.
    Moreover, for $\rho>\rho_*$ the map $\rho\mapsto m_\rho$ is strictly decreasing and strictly concave.
\end{lemma}
\proof
First, observe that $m_\rho>-\infty$ for all $\rho>0$ by \eqref{1051}.

For $\rho\le \rho_*$, by \eqref{GNCh} we have
$$
F(u)\ge\frac 12 |\nabla u|_2^2+\frac14|u|_4^4-\frac12\frac{\rho^{2}}{\rho_*^2}|\D u|_2^{2}>0\qquad\forall \rho\in \cS_\rho,\; 0<\rho\le \rho_*.
$$
Moreover,  consider the family in $\cS_\rho$
$$u_t(x)=\frac{\rho}{\rho_*}t^{3/2}w_*(t x),$$
then
$$
\lim_{t\to 0}F(u_t)=0.
$$
On the other hand, if $\rho>\rho_*$, taking into account \eqref{stima1} and \eqref{1123}, we see that
\begin{align*}
    F(u_t)&\le\frac12\Big(\frac{\rho}{\rho_*}\Big)^2|\nabla w_*|_2^2\, t^2+\frac{\bar c^4}{4}\rho\Big(\frac{\rho}{\rho_*}\Big)^4 |\nabla w_*|_2^3\,t^3-\frac14\Big(\frac{\rho}{\rho_*}\Big)^4\cN(w_*)\,t^2\\
    &=\frac12\left(\frac{\rho}{\rho_*}\right)^2|\n w_*|_2^2\left[1-\left(\frac{\rho}{\rho_*}\right)^2+\tilde c\, t\right]  t^2 <0,
\end{align*}
for all sufficiently small $t>0$.

{The proof of the continuity and concavity of $m_\rho$, for $\rho>\rho_*$, is the same as in the proof of Propositions \ref{PF}.}
\qed

\begin{lemma}\label{l-62}
    Let $\alpha=1$ and $\rho>0$.
    The constrained functional $F|_{\cS_\rho}$ has no critical points $u$ with $\D u\in H^1_{\loc}$ at nonnegative energy levels.
    Moreover, if $\mu_\rho<0$ is a critical level for $F$ on $\cS_\rho$ then the corresponding Lagrange multiplier $\lambda_\rho$ satisfies
    \beq\label{lambdaa1}
    \lambda_\rho\ge -\frac{4\mu_\rho}{\rho^2}.
    \eeq
\end{lemma}
\proof
Given $\rho>0$, let $\mu_\rho\in\R$ be a critical level for $F$ on $\cS_\rho$.
Recall that, by the system \eqref{abc-syst},
\beq\label{sys-a1}
A=\lambda_\rho\rho^2+4\mu_\rho,\quad B=-8 \mu_\rho,\quad C=2\lambda_\rho\rho^2 -4\mu_\rho,
\eeq
with the usual notations.
Then, if $\mu_\rho\ge 0$ we get $B\le 0$, a contradiction.

The first relation in \eqref{sys-a1}  holds without any auxiliary regularity assumption on the solution and provides \eqref{lambdaa1}, for $\mu_\rho<0$.
\qed

\medskip

\begin{prop}\label{Pmin-1}
    Let $\alpha=1$ and $\rho>\rho_*$.
    Then there exists a positive normalised solution $u_\rho\in\cS_\rho$ of \eqref{P} such that $F(u_\rho)=m_\rho<0$.
\end{prop}

\proof
Similar to the proof of Proposition \ref{Pmin}.
To verify that the minimizing sequences are bounded in $H^1$, we use \eqref{Sinn} in place of \eqref{stima1 F basso eq+}.
\qed

\subsection{Choquard limit}

\begin{lemma}
    \label{L5.4}
    Let $\alpha=1$ and $\rho>\rho_*$. Then
    $$
    0>m_\rho\ge -C \rho^4\left(1-\left(\frac{\rho_*}{\rho}\right)^2\right).
    $$
    In particular,
    $$m_\rho\gtrsim-(\rho-\rho_*)\to 0\quad\text{as $\rho\to\rho_*^+$},\qquad m_\rho\gtrsim-\rho^4\quad\text{as $\rho\to\infty$}.$$
\end{lemma}

\proof
Let $\rho>\rho_*$ and $u_\rho\in \cS_\rho$ be a minimizing function for $F_{|_{\cS_\rho}}$.

Taking into account \eqref{sys-a1}, and  $3|u_\rho|_4^4=2\cN(u_\rho)-4|\D u_\rho|_2^2$ by system \eqref{abc-syst}, then
$$
m_\rho=-\frac18 |u_\rho|_4^4=-\frac{(\cN(  u_\rho)-2|\n   u_\rho|_2^2)^3}{27|  u_\rho|_4^8}
\ge
-\frac{\left(1-\frac{\rho_*^2}{\rho^2}\right)\cN (u_\rho)^3}{c\frac{1}{\rho^4} \cN (u_\rho)^3}=
-C \rho^4\left(1-\left(\frac{\rho_*}{\rho}\right)^2\right),
$$
by \eqref{GNCh} and \eqref{stima-CB}.
\qed

\medskip

\begin{rem}
Additionally,  for $\alpha=1$ we have exactly $m_\rho\sim-\rho^4$ as $\rho\to\infty$, by the same arguments used to prove \eqref{eAs>1 infty}.
\end{rem}

\begin{lemma}\label{l-upper1}
    $\alpha=1$ and  $\rho>\rho_*$.
    Then $\lambda_\rho\lesssim |m_\rho|^\frac{2}{3}\lesssim(\rho-\rho_*)^{2/3}\to 0$, as $\rho\to\rho_*$.
\end{lemma}

\proof
Let $\rho>\rho_*$, $\rho\to\rho_*>0$,  then $m_\rho\to 0$ by Lemma \ref{L5.4}.
Let $u_\rho$ be the minimizer for $F$ on $\cS_\rho$.
From \eqref{stima-CB} and \eqref{sys-a1} we conclude that
$$
2\lambda_\rho\rho^2 -4m_\rho=\cN(u_\rho)
\le
 {c}_{1} \,\rho^{\frac{4}{3}}|u_\rho|_4^{\frac{8}{3}}=8^\frac23  {c}_{1} \,\rho^{\frac{4}{3}}|m_\rho|^\frac{2}{3}
$$
and
$$
0\le \lambda_\rho\rho^2
\le 2m_\rho+2 {c}_{1} \,\rho^{\frac{4}{3}}|m_\rho|^\frac{2}{3}\simeq 2 {c}_{1} \,\rho_*^{\frac{4}{3}}|m_\rho|^\frac{2}{3}\to 0,\quad\mbox{ as }\rho\to\rho_*^+
$$
as required.
\qed

\bigskip
Using the upper bound from Lemma \ref{l-upper1}, we deduce the convergence to the Choquard limit after a rescaling. Let $\alpha=1$ and $\rho\to\rho_*$.
For $u\in\cS_\rho$, consider the rescaling
\begin{equation}\label{resc-1}
    u(x)\quad\mapsto\quad v_{\lambda_\rho}(x):=\lambda_{\rho}^{-3/4}u\Big(\frac{x}{\sqrt{\lambda_\rho}}\Big)\in \cS_{\rho},
\end{equation}
and recall that $\lambda_\rho\lesssim(\rho-\rho_*)^{2/3}\to 0$.
Set
$$
\bar{F}_\rho(v):=E_\ch(v)+\tfrac{1}{4}\lambda_\rho^{1/2}|v|_4^4,\qquad v\in \cS_\rho.
$$
Notice that by \eqref{resc-1},
$$
F_\rho(v_{\lambda_\rho})=\lambda_\rho^{-1}F(u),\qquad u\in \cS_\rho.
$$
Therefore,
\beq
\label{UN}
\bar m_\rho:=\min_{\cS_\rho}\bar{F}_\rho=\lambda_\rho^{-1}m_\rho .
\eeq
Moreover, if $u_\rho\mapsto\bar v_\rho\in \cS_\rho$ is the rescaling \eqref{resc-1} of the global minimizer of $F$ on $\cS_\rho$, then $\bar v_\rho$ satisfies the Euler--Lagrange equation
\beq
\label{001}
-\Delta \bar v_\rho+\bar v_\rho + \lambda_\rho^{1/2}\bar v_\rho^3=(I_1*\bar v_\rho^2)\bar v_\rho\quad\text{in $\R^3$}.
\eeq

\begin{prop}
    Let $\alpha=1$ and $\rho>\rho_*$.
    Then $\lambda_\rho\to 0$ as $\rho\to\rho_*$, and  the rescaling
    $$\bar v_{\lambda_\rho}(x)=\lambda_{\rho}^{-3/4}u_{\rho}\Big(\frac{x}{\sqrt{\lambda_\rho}}\Big)$$
    of the minimizer $u_\rho$ converges (up to a subsequence) in $H^1$ and $L^4$ to a ground state $w_*>0$ of the Choquard equation \eqref{ChL2}.
\end{prop}

\proof
The claim $\lambda_\rho\to 0$ as $\rho\to\rho_*$ follows from Lemma \ref{l-upper1}.
We shall study the properties of $\bar v_{\lambda_\rho}$.
By construction,
$$|\bar v_{\lambda_\rho}|_2=|u_{\rho}|_2=\rho\to\rho_*.$$
Note also that (cf. \eqref{sys-a1})
$$|\nabla\bar v_{\lambda_\rho}|_2^2=\rho^2+4\frac{m_\rho}{\lambda_\rho},\quad \cN(\bar v_{\lambda_\rho})=2\rho^2 -4\frac{m_\rho}{\lambda_\rho}.$$
Then, taking into account that $m_\rho<0$,
\begin{equation}\label{min-GN}
    \frac{\rho_*^{2}}{2}\le
    \frac{|\bar v_{\lambda_\rho}|_2^2|\D \bar v_{\lambda_\rho}|_2^{2}}{\cN(\bar v_{\lambda_\rho})}=\frac{\rho^2\left(\rho^2+4\frac{m_\rho}{\lambda_\rho}\right)}{2\rho^2 -4\frac{m_\rho}{\lambda_\rho}}\le \frac{\rho^2}{2} \to \frac{\rho_*^{2}}{2}.
\end{equation}
In particular,
\beq
\label{e6.7}
\frac{m_\rho}{\lambda_\rho}=o(1)\quad\mbox{ as }\rho\to\rho_*^+,
\eeq
otherwise we would get a contradiction with the lower bound in \eqref{min-GN}.

Thus \eqref{min-GN} implies that $\bar v_{\lambda_\rho}$ is a minimizing sequence for the Gagliardo--Nirenberg quotient \eqref{GNCh}. By \cite[Lemma 2.3]{Ye},  $\bar v_{\lambda_\rho}$ has a subsequence that converges to a ground state $w_*$ of \eqref{ChL2} in $H^1$, and hence also in $L^4$.
Moreover, $w_*$ is a  solution of  \eqref{ChL2}, by \eqref{001}.
\qed

\begin{rem}
By \eqref{e6.7} and \eqref{UN} we can deduce that $\bar m_\rho\to 0$ as $\rho\to\rho_*$.
\end{rem}

Next we deduce an upper bound on $m_\rho$ and lower bound on $\lambda_\rho$, which however are weaker than the opposite bounds we already have.

\begin{lemma}
    \label{L5.upper}
    Let $\alpha=1$ and $\rho>\rho_*$.
    Then $m_\rho\lesssim-(\rho-\rho_*)^3$ and $\lambda_\rho\gtrsim -m_\rho\gtrsim(\rho-\rho_*)^3$ as $\rho\to\rho_*^+$.
\end{lemma}

\proof
For $\rho>\rho_*$, set
$$w_{\rho,t}:=\frac{\rho}{\rho_*}t^{3/2}w_*(t x)\in\cS_{\rho},$$
where $w_*\in\cS_{\rho_*}$ is a $1$-frequency minimizer for $m_{\rho_*}^\ch$.
Taking into account \eqref{1123}, we have
$$
F_\rho(w_{\rho,t})=-\frac{\rho^2_*}{2}\left(\frac{\rho}{\rho_*}\right)^2\left(\left(\frac{\rho}{\rho_*}\right)^2-1\right)\, t^2+\frac{|w_*|_4^4}{4} \left(\frac{\rho}{\rho_*}\right)^4\, t^3.
$$
Note that $F_\rho(w_{\rho,t})<0$ as $\rho>\rho_*$ and $t\to 0$.
Minimizing with respect to $t$ for a fixed $\rho>\rho_*$, we find that the minimum occurs at $t_\rho\sim \frac{\rho}{\rho_*}-1$, as $\rho\to\rho_*^+$,  and
$$m_\rho\le F_\rho(w_{\rho,{t_\rho}})\lesssim -\frac{\left(\Big(\frac{\rho}{\rho_*}\Big)^2-1\right)^3}{\Big(\frac{\rho}{\rho_*}\Big)^2}\sim -\left(\rho-\rho_*\right)^3,$$
as $\rho\to\rho_*$. Finally, we note that according to Lemma \ref{l-62},
\begin{equation*}
\lambda_\rho\ge -\frac{4m_\rho}{\rho^2}\gtrsim(\rho-\rho_*)^3,\quad\mbox{ as }\rho\to\rho_*^+.\qedhere
\end{equation*}

\begin{rem}
    Combining previous estimates together  we see that
    $$(\rho-\rho_*)^3\lesssim|m_\rho|\lesssim\lambda_\rho\lesssim|m_\rho|^{2/3}\lesssim (\rho-\rho_*)^{2/3}.$$
    It remains an open question to find matching upper and lower bounds on $m_\rho$.
\end{rem}

\subsection{Thomas--Fermi limit}
We simply observe that all arguments in the Subsection \ref{ssTF} and in the proof of Proposition \ref{pTFconv} remain valid for $\alpha=1$.

\appendix
\section{$L^2$--convergence of the fixed frequency ground states}

Recall that in  \cite[Theorem 2.5]{LZVM} it was proved that
as $\lambda\to 0$, the rescaled family of ground states $u_\lambda$ of the fixed frequency problem \eqref{Plambda},
$$
w_\lambda(x):=\lambda^{-\frac{2+\alpha}{4}}u_\lambda\big(\lambda^{-1/2}x\big),
$$
converges in $D^1\cap L^4$, up to a subsequence, to a positive radially symmetric ground state $w_*\in H^1\cap C^2$ of the {\em Choquard equation} \eqref{C1}. 
We can also assume that the convergence is almost everywhere.
We are going to show that these facts also imply the $L^2$--convergence.

\begin{lemma}\label{l-App-L2}
In the above notations, $|w_\lambda-w_*|_2\to 0$ as $\lambda\to 0$, up to a subsequence.
\end{lemma}

\proof
The rescaled ground states $w_\lambda\in H^1\cap W^{2,2}\cap C^2$ solve
\begin{equation}\label{C1+}
	-\Delta w+w+\lambda^{\alpha/2}|w|^2w=(I_{\alpha}*|w|^2)w\quad\text{in $\R^3$},
\end{equation}
and the rescaled energy takes the form
$$
 \bar F_\lambda(w):=\frac{1}{2}\int_{\R^3} |\D w|^2dx+\frac{1}{2}\int_{\R^3} | w|^2dx-\frac{1}{4}\cN(w)+\frac{\lambda^{\alpha/2}}{4}\int_{\R^3} |w|^4dx.
$$
Denote
\begin{equation*}\label{abc+lambda}
	A_\lambda=|\nabla w_\lambda|_2^2,\quad B_\lambda=|w_\lambda|_4^4,\quad C_\lambda=\cN(w_\lambda),\quad \bar\rho_\lambda^2=|w_\lambda|_2^2,\quad \bar\mu_\lambda=\bar F_\lambda(w_\lambda)
\end{equation*}
and
\begin{equation*}\label{abc*}
	A_*=|\nabla w_*|_2^2,\quad  \quad C_*=\cN(w_*),\quad \bar\rho_*^2=|w_\lambda|_2^2,\quad \bar\mu_*=\bar E_\ch (w_*).
\end{equation*}
From the Energy--Nehari--Poho\v zaev relations \cite[Proposition 4.1]{LZVM} we deduce for $w_\lambda$ the algebraic system
\begin{equation*}\label{abc-syst-lambda+}
	\left\{
	\begin{aligned}
		\frac12A_\lambda+\frac12\bar\rho_\lambda^2+\frac14\lambda^{\alpha/2} B_\lambda-\frac1{4} C_\lambda&=\bar\mu_\lambda\\
		A_\lambda+\bar\rho_\lambda^2+\lambda^{\alpha/2}B_\lambda-C_\lambda&=0,\\
		\frac12 A_\lambda+\frac32\bar\rho_\lambda^2+\frac34\lambda^{\alpha/2}B_\lambda-\frac{3+\alpha}{4}C_\lambda&=0,
	\end{aligned}
	\right.
\end{equation*}
and an analogous system for $w_*$, without the $B_*$-term.
In particular, we conclude that
\begin{equation}\label{abc-syst-plus-lambda_T}
	\bar\rho_\lambda^2=\frac{1+\alpha}{3-\alpha}A_\lambda+\frac{\alpha}{3-\alpha}\lambda^{\alpha/2}B_\lambda,\qquad
	 \bar\rho_*^2=\frac{1+\alpha}{3-\alpha}A_*.
	 \end{equation}	 
Since
\beq
\label{TO}
A_\lambda\simeq|\nabla w_*|_2^2,\qquad B_\lambda=|w_*|_4^4,
\eeq
by \cite[Theorem 2.5]{LZVM}, we conclude that $\{|w_\lambda|_2\}_\lambda$ is bounded and hence $w_\lambda$ weakly converges to $w_*$, up to a subsequence.
Then by \eqref{TO} and \eqref{abc-syst-plus-lambda_T}, 
\begin{equation*}\label{abc-bar}
\bar\rho_\lambda^2\ \longrightarrow\ \frac{1+\alpha}{3-\alpha}|\nabla w_*|_2^2=\bar\rho_*^2\quad\text{as $\lambda\to 0$},
\end{equation*}
that is $|w_\lambda |_2\to  |w_*|_2$. 
We conclude that $|w_\lambda-w_*|_2\to 0$ as $\lambda\to 0$, up to a subsequence.
\qed

\bigskip
\subsection*{Acknowledgements}
{\small R.M. has been supported by the ``Gruppo Nazionale per l'Analisi Matematica, la Probabilit\`a e le loro Applicazioni (GNAMPA)'' of the {\em Istituto Nazionale di Alta Matematica (INdAM)} and by the MIUR Excellence Department Project MatMod@TOV  awarded to the Department of Mathematics, University of Rome Tor Vergata, CUP E83C23000330006. G.R. has been supported by the ``Gruppo Nazionale per l'Analisi Matematica, la Probabilit\`a e le loro Applicazioni (GNAMPA)'' of the {\em Istituto Nazionale di Alta  Matematica (INdAM)} and by PRIN project 2017JPCAPN (Italy): {\em Qualitative and quantitative aspects of nonlinear PDEs}.}



\end{document}